\documentclass[reqno,a4paper,12pt]{amsart}

\numberwithin{equation}{section}
\usepackage[utf8]{inputenc}
\usepackage{amsmath, amsthm, amssymb, amscd, accents, bm, amsfonts}
\usepackage{mathtools}
\usepackage{url}
\usepackage{mathrsfs,dsfont}
\usepackage{datetime}
\usepackage{hyperref}
\usepackage{colonequals}
\usepackage{enumerate}
\usepackage[sort,nocompress]{cite}

\usepackage[pdftex,dvipsnames,usenames]{color}

\mathtoolsset{showonlyrefs}

\newtheorem{definition}{Definition}[section]
\newtheorem{theorem}[definition]{Theorem}
\newtheorem*{theorem*}{Theorem}
\newtheorem{lemma}[definition]{Lemma}
\newtheorem{corollary}[definition]{Corollary}
\newtheorem{proposition}[definition]{Proposition}

\def\N{{\mathbb N}}
\def\Z{{\mathbb Z}}
\def\R{{\mathbb R}}
\def\T{{\mathbb T}}
\def\C{{\mathbb C}}
\def\Q{{\mathbb Q}}

\newcommand{\Rd}{{\R^d}}

\newcommand{\lspan}{{\mathrm{span}}}
\newcommand{\supp}{{\mathrm{supp}}}

\newcommand\1{\mathds{1}}

\newcommand{\ift}{{\mathcal{F}^{-1}}}
\newcommand{\ft}{{\mathcal{F}}}

\DeclarePairedDelimiter\floor{\lfloor}{\rfloor}
\allowdisplaybreaks

\newcommand{\norm}[1]{\left\lVert #1 \right\rVert}

\makeatletter
\renewcommand{\@secnumfont}{\bfseries}
\makeatother

\RequirePackage{listings}
\RequirePackage{needspace}
\RequirePackage{color}
\providecommand{\lean}[1]{\texttt{\detokenize{#1}}}

\definecolor{LeanKeyword}{rgb}{0.12,0.29,0.52}
\definecolor{LeanOperator}{rgb}{0.14,0.34,0.63}
\definecolor{LeanPunctuation}{rgb}{0.76,0.30,0.06}

\lstdefinestyle{leanpaper}{
  basicstyle=\ttfamily\fontsize{9pt}{11pt}\selectfont,
  keywordstyle=\color{LeanKeyword}\bfseries,
  morekeywords={def,theorem,variable,fun,let,by,if,then,else},
  columns=fullflexible,
  keepspaces=true,
  showstringspaces=false,
  breaklines=true,
  breakatwhitespace=false,
  aboveskip=.6\baselineskip,
  belowskip=.6\baselineskip,
  literate=
    {:=}{{{\color{LeanPunctuation}:=}}}2
    {=>}{{{\color{LeanPunctuation}=>}}}2
    {(}{{{\color{LeanPunctuation}(}}}1
    {)}{{{\color{LeanPunctuation})}}}1
    {:}{{{\color{LeanPunctuation}:}}}1
    {,}{{{\color{LeanPunctuation},}}}1
    {+}{{{\color{LeanOperator}+}}}1
    {-}{{{\color{LeanOperator}-}}}1
    {*}{{{\color{LeanOperator}*}}}1
    {/}{{{\color{LeanOperator}/}}}1
    {0}{{{\color{LeanOperator}0}}}1
    {1}{{{\color{LeanOperator}1}}}1
    {2}{{{\color{LeanOperator}2}}}1
    {3}{{{\color{LeanOperator}3}}}1
    {4}{{{\color{LeanOperator}4}}}1
    {5}{{{\color{LeanOperator}5}}}1
    {6}{{{\color{LeanOperator}6}}}1
    {7}{{{\color{LeanOperator}7}}}1
    {8}{{{\color{LeanOperator}8}}}1
    {9}{{{\color{LeanOperator}9}}}1
    {¬}{{{\color{LeanOperator}$\neg$}}}1
    {Λ}{{$\Lambda$}}1
    {α}{{$\alpha$}}1
    {β}{{$\beta$}}1
    {δ}{{$\delta$}}1
    {ε}{{$\varepsilon$}}1
    {η}{{$\eta$}}1
    {μ}{{$\mu$}}1
    {ξ}{{$\xi$}}1
    {π}{{$\pi$}}1
    {ρ}{{$\rho$}}1
    {ᵈ}{{\textsuperscript{d}}}1
    {ᵐ}{{\textsuperscript{m}}}1
    {‖}{{$\Vert$}}1
    {⁻}{{\textsuperscript{-}}}1
    {₀}{{\textsubscript{0}}}1
    {₂}{{\textsubscript{2}}}1
    {ℂ}{{$\mathbb C$}}1
    {ℕ}{{$\mathbb N$}}1
    {ℚ}{{$\mathbb Q$}}1
    {ℝ}{{$\mathbb R$}}1
    {ℤ}{{$\mathbb Z$}}1
    {↑}{{$\uparrow$}}1
    {→}{{{\color{LeanOperator}$\to$}}}1
    {∀}{{{\color{LeanOperator}$\forall$}}}1
    {∂}{{$\partial$}}1
    {∃}{{{\color{LeanOperator}$\exists$}}}1
    {∈}{{{\color{LeanOperator}$\in$}}}1
    {∉}{{$\notin$}}1
    {∑}{{{\color{LeanOperator}$\sum$}}}1
    {√}{{$\surd$}}1
    {∞}{{$\infty$}}1
    {∧}{{{\color{LeanOperator}$\land$}}}1
    {∩}{{$\cap$}}1
    {∫}{{$\int$}}1
    {≠}{{$\ne$}}1
    {≤}{{{\color{LeanOperator}$\le$}}}1
    {≥}{{{\color{LeanOperator}$\ge$}}}1
    {⊆}{{{\color{LeanOperator}$\subseteq$}}}1
    {⟨}{{$\langle$}}1
    {⟩}{{$\rangle$}}1
    {��}{{$\mathcal F$}}1
}

\lstnewenvironment{leancode}{\lstset{style=leanpaper}}{}

\begin{document}

\title[Universal completeness of exponentials]{Universal completeness of exponentials}

\author[S.~Bertolini]{Susanna Bertolini}
\address{ Department of Mathematics, ETH Z\"urich, Ramistrasse 101, 8092 Z\"urich, Switzerland}
\email{susanna.bertolini@math.ethz.ch}

\author[E.~Florit-Simon]{Enric Florit-Simon}
\address{ Department of Mathematics, ETH Z\"urich, Ramistrasse 101, 8092 Z\"urich, Switzerland}
\email{enric.florit@math.ethz.ch}

\author[L.~Liehr]{Lukas Liehr}
\address{Department of Mathematics, Bar-Ilan University, Ramat-Gan 5290002, Israel}
\email{lukas.liehr@biu.ac.il}

\author[M.~Taylor]{Mitchell A. Taylor}
\address{Mathematical Institute, University of Oxford, Andrew Wiles Building, Radcliffe Observatory Quarter, Woodstock Road, Oxford, OX2 6GG, United Kingdom}
\email{mitchtaylor@shaw.ca}

\date{\today}

\begin{abstract}
We consider generalizations of the classical Fourier uniqueness theorem. First, we construct a family of uniformly discrete sets $\Lambda \subset \mathbb{R}$, of uniform density one, such that the exponential system $\{e^{2\pi i\lambda x} : \lambda \in \Lambda\}$ is complete in $L^p(S)$ for every $1 \leq p < \infty$ and every measurable set $S \subset \mathbb{R}$ with $|S| < 1$. We also show that no set that is asymptotically integer can have this universality property. Additionally, for every $v \in (0,1)$, we construct a set of integer frequencies and uniform density $v$ whose exponential system is complete in $L^p(S)$ for every $1 \leq p < \infty$ and every measurable set $S \subset [0,1]$ with $|S| < v$.  Finally, we prove that the Sobolev regularity condition $\alpha > \frac12$ for the existence of uniformly discrete uniqueness sets for spectra with periodic weak gaps, considered by Olevskii and Ulanovskii, is sharp. Our findings admit extensions to higher dimensions and are verified in Lean.
\end{abstract}

\maketitle

\section{Introduction and results}\label{sec:intro}

The classical Fourier uniqueness theorem completely characterizes a function $f \in L^2(0,1)$ in terms of its Fourier coefficients $\hat f(n)=\int_0^1 f(t) e^{-2\pi i n t} \, dt$. Equivalently, the exponential system $\{ e^{2\pi i n x} : n \in \Z \}$ is complete in the space $L^2(0,1)$, i.e., the linear span of $\{ e^{2\pi i n x} : n \in \Z \}$ is dense in $L^2(0,1)$.

It is natural then to consider the completeness properties of more general exponential systems
$$
E(\Lambda) = \{ e^{2\pi i \lambda x} : \lambda \in \Lambda \},
$$
where $\Lambda\subset \R$ is not necessarily an arithmetic progression anymore. One often considers $\Lambda$ which are \textit{uniformly discrete}, namely with the property that $\inf \{ |\lambda - \lambda'| : \lambda, \lambda' \in \Lambda, \, \lambda \neq \lambda' \} > 0$. Additionally, one says that $\Lambda$ is \textit{uniformly distributed}, with \textit{uniform density} $0 \leq D(\Lambda) < \infty$, if
$$
\#\big ( \Lambda \cap [x,x+r] \big) = D(\Lambda)r + o(r), \quad r \to \infty,
$$
where $\#(\Omega)$ denotes the cardinality of a set $\Omega$ and the convergence is uniform with respect to $x \in \R$. An extension of the Fourier uniqueness theorem, which is based on complex-variable techniques, asserts that if $\Lambda$ is uniformly distributed and has uniform density $D(\Lambda) > 1$, then $E(\Lambda)$ is still complete in $L^2(0,1)$. Conversely, if $\Lambda$ is uniformly distributed and $E(\Lambda)$ is complete in $L^2(0,1)$, then necessarily $D(\Lambda) \geq 1$; see, for instance, \cite[Lectures 16-18]{Levin1996}, \cite{Young2001}, and \cite{Redheffer1977} for a survey. We also note a deep theorem of Beurling and Malliavin \cite{BeurlingMalliavin1967} which extends this to arbitrary sets $\Lambda$: it provides a large class of $\Lambda$ such that $E(\Lambda)$ is complete in $L^2(0,1)$, namely those for which their Beurling-Malliavin density exceeds the critical value $1$. Below the critical density, the system is incomplete. For an exposition on the significance of the Beurling-Malliavin theorem we refer to \cite{mashreghi2006beurling}.

\subsection{Completeness on families of sets}
The results above consider only the fixed set $S=[0,1]$, and aim at finding all $\Lambda$ such that $E(\Lambda)$ is complete in $L^2(S)$ (for corresponding characterizations of Riesz bases and frames for $L^2(0,1)$, see \cite{pavlov1979basicity,lyubarskii1997complete} and \cite{ortega2002fourier}). The previous question can be reversed by fixing $\Lambda$ and asking for a large class of measurable sets $S \subseteq \R$ of finite measure such that $E(\Lambda)$ is complete in $L^2(S)$.  In this setting, $S$ is usually referred to as the spectrum. 

If $E(\Lambda)$ is the classical exponential system, i.e., $\Lambda = \Z$, the sets $S$ for which $E(\Lambda)$ is complete in $L^2(S)$ are characterized: they are exactly those satisfying the geometric packing condition
$$
\sum_{k \in \Z} \1_S(x+k) \leq 1 \quad \text{for almost every } x \in \R,
$$
that is, the integer translates of $S$ are pairwise disjoint up to sets of measure zero. This follows from a periodization argument, see \cite[Section 8]{OlevskiiUlanovskii2020} and \cite{de2019three}. In particular, the completeness of $E(\Z)$ in $L^2(S)$ forces $|S| \leq 1$.

If one considers more general sets $\Lambda$, such as perturbations of the integers, then an exponential system can be complete in $L^2(S)$ for sets whose measure can be arbitrarily large instead. This is known as Landau's phenomenon. Precisely, Landau \cite{Landau1964} constructed a uniformly discrete set $\Lambda$ of uniform density one, such that $E(\Lambda)$ is complete in $L^2(S)$ for every set $S$ belonging to the class
$$
\Big \{ \bigcup_{k \in F} \, [k+a, k+1-a] \ :  F \subseteq \Z \text{ finite}, \ 0 < a < \tfrac12 \Big \}.
$$
The property that the intervals are at positive distance to the integers is crucial for this result: it guarantees that the so-called projection
$$
\mathrm{Proj}(S) := (S+\Z) \cap [0,1]
$$
omits a set of positive measure in $[0,1]$, while the number of integer translates of $S$ covering a given point in $[0,1]$ may be arbitrarily large. Moreover, the phenomenon is not tied to a particular construction: for every perturbation $\Lambda = \{ n + \delta_n : n \in \Z\}$ of the integers with $\delta_n \neq 0$ and $\delta_n \to 0$ as $|n| \to \infty$, the system $E(\Lambda)$ is complete in $L^2(S)$ for a certain class of sets of arbitrarily large measure \cite{Olevskii1997,bruna2006completeness}. A multi-dimensional version of Landau's phenomenon was obtained in \cite{Ulanovskii2001}. For a systematic treatment of completeness and sampling problems for disconnected spectra, we refer to the lecture notes \cite{OlevskiiUlanovskii2016} and the survey \cite{OlevskiiUlanovskii2020}. For a use of Landau's phenomenon to obtain new results on completeness problems of discrete translates we refer to a recent paper by Lev \cite{lev2025completeness}.

\subsection{Universal completeness on finite measure sets}
Landau's result requires $S$ to have a special structure. On the other hand, Olevskii and Ulanovskii \cite{OlevskiiUlanovskii2008} constructed a set $\Lambda$ of uniform density one such that $E(\Lambda)$ is complete in $L^2(S)$ for every bounded set $S$ with measure $|S| < 1$. One may take $\Lambda = \{ n + 2^{-|n|} : n \in \Z \}$; in fact, every perturbation of $\Z$ by a nonvanishing exponentially decaying sequence works \cite[Theorem~3.1]{OlevskiiUlanovskii2008}.

If the boundedness assumption is dropped (thus only $|S| < 1$ is assumed), Olevskii and Ulanovskii \cite{OlevskiiUlanovskii2011} constructed for each individual set $S$ a set $\Lambda_S$ (depending on $S$) of uniform density $D(\Lambda_S)=|S|$ such that $E(\Lambda_S)$ is complete in $L^2(S)$.\\
In this line of research, several natural problems have remained open:
\begin{itemize}
    \item The first is whether the two results above can be combined: does there exist a uniformly discrete set $\Lambda$ with uniform density $1$ such that $E(\Lambda)$ is complete in $L^2(S)$ for every $S$ with $|S|<1$? See \cite[Section 7.2]{OlevskiiUlanovskii2017} and \cite[Section 12]{OlevskiiUlanovskii2020}, where this problem is stated explicitly.
    \item A second problem, posed in \cite[Section~7.2, Question~3]{OlevskiiUlanovskii2017}, concerns uniqueness in an $L^1$-sense: given $S\subset\R$ with $|S| < 1$, does there exist a uniformly discrete set $\Lambda_S$ such that for every $f \in L^1(S)$ with
$$
\int_S f(x) e^{-2\pi i \lambda x} \, dx = 0, \quad \lambda \in \Lambda_S,
$$
one has $f=0$?
\end{itemize}
Our first main result gives an answer to both questions.
Hereinafter, $\{t\} = t - \floor{t}$ denotes the fractional part of $t \in \R$.

\begin{theorem}\label{thm:universal}
    Let $\alpha \in \R \setminus \Q$, and let $\beta\in \mathbb Q$ with $0 < \lvert \beta\rvert <\frac12$. Define the set
    $$
    \Lambda = \Lambda_{\alpha,\beta}=\{ n+ \delta_n : n \in \Z \}, \quad \mbox{with}\quad  \delta_n= \beta \{n\alpha\},
    $$
    which is uniformly discrete and has uniform density $D(\Lambda) = 1$.
    
 Then, for every $S \subset \R$ with $|S| < 1$ and every $f \in L^1(S)$, the condition
    $$
    \int_S f(x) e^{-2\pi i x \lambda} \, dx = 0, \quad \lambda \in \Lambda,
    $$
    implies that $f=0$. In particular, the exponential system $E(\Lambda)$ is complete in $L^p(S)$ for every $p \in [1,\infty)$ and every $S \subseteq \R$ with $|S| < 1$.
\end{theorem}

We note that Theorem \ref{thm:universal} determines directly an explicit family of sets $\Lambda_{\alpha,\beta}$ such that $E(\Lambda_{\alpha,\beta})$ is complete in $L^p(S)$ for every $p \in [1,\infty)$ and every $S \subseteq \R$ with $|S| < 1$. The sets $\Lambda_{\alpha,\beta}$ arise from the cut-and-project scheme that produces Meyer's simple quasicrystals. Sets of this type are known to have the following sampling set property: Matei and Meyer \cite{MateiMeyer2008,MateiMeyer2010} proved that every simple quasicrystal of density $d$ is a set of stable sampling for band-limited functions whose spectrum is a \emph{compact} set of measure less than $d$. The set $\{ n + \{n\alpha\} : n \in \Z\}$ appears there, in connection with a question of Golse \cite{MateiMeyer2010}. We refer to \cite{GrepstadLev2014,GrepstadLev2018} for the connection between universal sampling at the critical density and bounded remainder sets. A characterization of exponential frames at the critical density was recently established in \cite{enstad2025exponential}.

We also note that the existence problem for frames on unbounded spectra was settled by Nitzan, Olevskii, and Ulanovskii \cite{NitzanOlevskiiUlanovskii2016}, who proved that every measurable set $S\subseteq\R$ with $0<|S|<\infty$ admits a frame of exponentials. More recently, Bownik and van Velthoven \cite{BownikVanVelthoven2025Redundancy} showed that, for every such $S$ and every $\varepsilon>0$, the frequency set to obtain an exponential frame can be chosen with upper Beurling density at most $(1+\varepsilon)|S|$. In both results, the frequency set depends on $S$.

\subsection{Non-completeness of asymptotically integer frequencies}
The set $\Lambda_{\alpha,\beta}$ in Theorem \ref{thm:universal} is a perturbation of the integers $\Z$ by a sequence $\delta_n=\beta \{n\alpha\}$ satisfying $\delta_n\not \to 0$. This a key structural difference compared to the sequences considered by Olevskii and Ulanovskii in \cite{OlevskiiUlanovskii2008} where $\delta_n \to 0$. This is not accidental: already in \cite[Example 3.2]{OlevskiiUlanovskii2008} it was shown that there exist unbounded sets $S$ of arbitrarily small measure such that completeness of $E(\Lambda)$ in $L^2(S)$ fails for the perturbed integers $\Lambda = \{\pm(n+2^{-k(n)})\}$.

Our next theorem extends this to arbitrary sequences that are asymptotically integer: a perturbation with $\delta_n\to 0$ can never yield a universal completeness set of exponentials, even on sets of arbitrarily small measure.

\begin{theorem}\label{thm:null}
    Let $\Lambda = \{ n+ \delta_n : n \in \Z \}$, with $\delta_n\to 0$ as $|n|\to\infty$. For every $\varepsilon>0$, there exists a measurable $S \subseteq \R$ such that $|S| < \varepsilon$ and $E(\Lambda)$ is not complete in $L^2(S)$.
\end{theorem}

\subsection{Universal completeness of integer frequencies}

We next consider the situation where $S$ is a measurable subset of the unit interval $[0,1]$. After rescaling, Theorem \ref{thm:universal} implies that for every $v \in (0,1]$ there exists a uniformly discrete set $\Lambda_v$ with uniform density $D(\Lambda_v)=v$ such that $E(\Lambda_v)$ is complete in $L^2(S)$ for every $S \subseteq [0,1]$ with $|S|<v$. The frequencies obtained in this way are, however, never integers. Olevskii and Ulanovskii studied in \cite{OlevskiiUlanovskii2008} the question whether, in this setting, a universal completeness set of exponentials $E(\Lambda_v)$ for all sets $S$ with $|S|<v$ can be chosen in a uniformly distributed way with the additional arithmetic property that $\Lambda_v$ is a subset of $\Z$. Questions of this nature go back to the work of Bourgain and Tzafriri \cite{BourgainTzafriri1987}.

It is shown in \cite{OlevskiiUlanovskii2008} that for every value $v$ belonging to the set
$$
\mathcal{V} = \left \{1-\frac1m : m \in \N_{>1}  \right \}
$$
there is a set $\Lambda_v \subseteq \Z$ of uniform density $v$ such that $E(\Lambda_v)$ is complete in $L^2(S)$ for every $S \subseteq [0,1]$ with $|S| < v$ (for $v=\frac12$ one may even allow $|S| = \frac12$), see \cite[Theorems 2.2 and 2.4]{OlevskiiUlanovskii2008}. In addition, it is shown that randomly chosen subsets of $\Z$ of the same density almost surely fail to have this universality property \cite[Theorem 2.5]{OlevskiiUlanovskii2008}, so the arithmetic structure of $\Lambda_v$ is essential. 

Observe that the values of $\mathcal{V}$ all lie in $[\frac12,1)$, the right half of the unit interval. Indeed, the technique of \cite{OlevskiiUlanovskii2008}, which rests on a Hardy space argument, is designed for densities $v \geq \frac12$ and, to the best of the authors' knowledge, does not extend to values below $\frac12$.

For all other values of $v$, the existence of a universal completeness set of exponentials with integer frequencies remained an open problem. In particular, it is conjectured in \cite[p. 59]{OlevskiiUlanovskii2016} that one cannot obtain such a result for values $v < \frac12$, and \cite[Open Question 6.5]{OlevskiiUlanovskii2016} and \cite[Section 12]{OlevskiiUlanovskii2020} pose the analogous problem with $v=\frac{1}{3}$.

Our next result settles the problem of universal completeness of integer frequencies for all subsets of the unit interval with $|S|<v$ as well as in the general $L^p$-sense, answering in particular the questions above.
\begin{theorem}\label{thm:integer}
    Let $\alpha \in \R \setminus \Q$, let $v \in [0,1]$, and let $\Lambda_v \subseteq \Z$ be defined by
    $$
    \Lambda_v=\{n\in\Z:\{n\alpha\}\in[1-v,1)\},
    $$
    which has uniform density $D(\Lambda)=v$. Then for every $S \subseteq [0,1]$ with $|S|<v$ and every $f \in L^1(S)$, the condition
    $$
    \int_S f(x) e^{-2\pi i x \lambda} \, dx = 0, \quad \lambda \in \Lambda_v,
    $$
    implies that $f=0$. In particular, the exponential system $E(\Lambda_v)$ is complete in $L^p(S)$ for every $p \in [1,\infty)$ and every $S \subseteq [0,1]$ with $|S| < v$.
\end{theorem}

The sets $\Lambda_v$ also appear in the work of Matei and Meyer \cite{MateiMeyer2009}, who proved that $E(\Lambda_v)$ is a frame for $L^2(S)$ whenever $S\subseteq[0,1]$ is compact and $|S|<v$. The same frequency sets were studied by Kozma and Lev \cite{KozmaLev2011} at the critical density. In particular, if $v\in(0,1)\cap(\Z+\alpha\Z)$, then $E(\Lambda_v)$ is a Riesz basis for $L^2(S)$ whenever $S\subseteq[0,1]$ is a finite union of disjoint intervals of total measure $v$, each of whose lengths belongs to $\Z+\alpha\Z$. 

\subsection{Uniqueness sets for periodic weak gaps}
The preceding results admit an equivalent formulation in the language of uniqueness sets for Paley-Wiener spaces. Given $S \subseteq \R$, the Paley-Wiener space $PW_S$ is defined by
$$
PW_S = \big \{ \ft F : F \in L^2(\R), \ F = 0 \text{ a.e. on } \R \setminus S \big \},
$$
where $\ft F(x) = \int_\R F(t) e^{-2\pi i x t} \, dt$ denotes the Fourier transform of $F$.
If $S$ has finite measure, then $L^2(S) \subseteq L^1(S)$ and every element of $PW_S$ is continuous. Moreover $\langle F, e^{2\pi i \lambda \, \cdot} \rangle_{L^2(S)} = \ft F(\lambda)$ for all $F \in L^2(S)$ and $\lambda \in \R$. Consequently, for $|S| < \infty$ we have the equivalence
$$
E(\Lambda) \text{ is complete in } L^2(S) \quad \Longleftrightarrow \quad \Lambda \text{ is a uniqueness set for } PW_S.
$$
For example, in this language, Theorem \ref{thm:universal} yields a uniformly discrete set $\Lambda_{\alpha,\beta}$ of density one which is a uniqueness set for all spaces $PW_S$ with $|S| < 1$. It is natural to consider whether similar uniqueness properties can hold also when $S$ has infinite measure. In this case, $PW_S$ contains discontinuous functions and one is led to consider the subspace $PW_S \cap C(\R)$, or to impose additional regularity (as, otherwise, the point-evaluation $f(\lambda)$ is no longer meaningful).

For $\alpha \geq 0$, define the Sobolev based Paley-Wiener spaces $PW_S^{(\alpha)}$  by
$$
PW_S^{(\alpha)} = \left \{ f = \ft F \in PW_S : \int_S (1+|t|^{2\alpha}) |F(t)|^2 < \infty \right \},
$$
so that $PW^{(0)}_S = PW_S$. If $\alpha > \frac12$, then $F \in L^1(\R)$ by the Cauchy-Schwarz inequality, and therefore $PW^{(\alpha)}_S \subseteq C(\R)$.
Following \cite{OlevskiiUlanovskii2017}, we say that $S$ has a periodic weak gap if there is $A\subset [0,1]$ with $|A|<1$ such that $S\subset A+\Z$. Notice that this in particular implies that $|\mathrm{Proj}(S)| < 1$.

Olevskii and Ulanovskii showed in \cite[Theorem 3]{OlevskiiUlanovskii2017} that if $S$ has periodic weak gaps, then $PW^{(\alpha)}_S$ admits a uniformly discrete uniqueness set for every $\alpha > \frac12$. The periodic structure is moreover crucial: for sets $S$ with randomly spaced gaps, every uniformly discrete set is a nonuniqueness set for $PW^{(\alpha)}_S$ \cite[Section 6]{OlevskiiUlanovskii2017}.

Based on their findings, they posed the following question \cite[Section 7.2]{OlevskiiUlanovskii2017}: if $S$ has periodic weak gaps, does the space $PW_S \cap C(\R)$ admit a uniformly discrete uniqueness set? More generally, one can ask if the regularity condition $\alpha > \frac12$ can be weakened. Our final result shows that the Sobolev exponent $\frac12$ is a sharp threshold, giving a negative answer to these questions.

\begin{theorem}\label{thm:sobolev}
Let $A\subseteq[0,1]$ be measurable with $0<|A|<1$,
set $S=A+\Z$, and let $\alpha\geq0$.
The following statements are equivalent:
\begin{enumerate}
    \item The space $PW^{(\alpha)}_S\cap C(\R)$ admits
    a uniformly discrete uniqueness set.
    \item $\alpha>\frac12$.
\end{enumerate}
\end{theorem}

Notice that the implication $(2) \implies (1)$ is the theorem of Olevskii and Ulanovskii \cite[Theorem 3]{OlevskiiUlanovskii2017}. The new content of Theorem \ref{thm:sobolev} is the reverse direction: at and below the endpoint $\alpha = \frac12$ no uniformly discrete uniqueness set exists. In particular, choosing $\alpha = 0$ and any $A$ as above yields a negative answer to the endpoint question described above.

\begin{corollary}\label{cor:OU}
    There exists a measurable set $S \subseteq \R$ with periodic weak gaps such that $PW_S \cap C(\R)$ does not admit a uniformly discrete uniqueness set. One may take $S = [0,\frac12] + \Z$.
\end{corollary}

\subsection{Higher dimensions}\label{sec:highdim}
All of our results admit direct extensions to higher dimensions, and the proofs are straightforward generalizations of the ones in the one-dimensional case. They do however require messier notation, which we believe obscures the key arguments in comparison. For this reason, we provide only their statements in Section \ref{sec:higher-dim}. The companion Lean verifications are in any case done directly in higher dimensions, for the sake of providing a complete proof of correctness. A discussion of the Lean formalization can be found in Section \ref{sec:lean}.

\subsection*{Usage of Large Langauge Models}\label{sec:llmusageandlean}

Large Language Models used under the author's guidance were employed as a research tool in the development of this work. For Theorem~\ref{thm:sobolev}, the authors directed GPT-5.4 toward a construction using successive
interpolation corrections and suggested treating the critical exponent by averaging functions with disjoint frequency supports.

Next, we focused on the problem of universal complete systems of exponentials. We started with the question of existence of a uniformly discrete set $\Lambda$ such that $E(\Lambda)$ is complete in $L^2(S)$ for every measurable set $S\subset\R$ with $|S|<1$. Initial work with GPT-5.5 (both Chat and Codex versions), continued with GPT-5.6, focused on proving that no such set exists. This led to Theorem~\ref{thm:null}, which shows that perturbations of the integers tending to zero cannot have this universality property. When this approach yielded no further progress, we reconsidered the problem: we instead attempted to prove universality for a set that had emerged as an apparent obstruction to generalizing the negative approach. In our present notation, this was the set $\Lambda_{\alpha,1/10}$ with $\alpha=\frac{1+\sqrt{5}}{2}$, which was obtained by GPT-5.6.

The resulting proof was subsequently extended to all irrational $\alpha$, after the authors observed that the argument did not depend on this particular choice. The generalization using the parameter $\beta$ was designated by the authors.

In parallel, the authors suspected that the same sets might also give uniqueness in the weaker $L^1$ setting, instead of requiring $L^2$ integrability. This extension was confirmed and executed using GPT-5.6.

Putting all the above together, this gives rise to Theorem~\ref{thm:universal}, which answers both questions (universal completeness and uniqueness in an $L^1$-sense) at the same time.

Finally, we discuss Theorem~\ref{thm:integer}. The starting point was an unpublished result of the third author establishing universal completeness
for densities $$v\in\left\{1-\tfrac{k}{p}:
p\text{ prime},\ k\in\mathbb N,\ 1\leq k\leq\tfrac{p}{2}\right\} \in [\tfrac12,1).$$ Initial investigations used GPT-5.4, GPT-5.5, and GPT-5.6. Since the Hardy space approach for proving the result of the third author did not extend to all densities $v \in [\frac12,1)$, the authors directed GPT-5.6 to explore a cut-and-project construction considered in \cite{MateiMeyer2009,KozmaLev2011} that is flexible enough to cover all densities $v \in [\frac12,1)$. The argument was subsequently extended to cover the entire range $v\in[0,1]$, yielding Theorem~\ref{thm:integer}.

All AI-generated material obtained in the process discussed above were modified and improved significantly, with the aim of shortening them, clarifying the intuition behind them, and highlighting the main ideas. This was achieved combining human ideas and further use of AI tools: GPT-5.6, and Claude Fable 5 and Claude Fable 5.1.

Finally, the Lean formalization was conducted with the help of Codex-5.5 and Codex-5.6.

\section{Universal completeness on finite measure sets}
\label{sec:L1}

In this section we prove Theorem~\ref{thm:universal}. Let
$\alpha\notin\Q$ and $\beta\in\Q$ with $0 < |\beta|<1/2$.
We use the centered perturbations
\begin{equation}\label{eq:seq}
    \delta_n=\beta\left(\{n\alpha\}-\frac12\right),
 \quad \lambda_n=n+\delta_n,\quad \Lambda=\{\lambda_n:n\in\Z\}.
\end{equation}
The frequencies in \eqref{eq:seq} differ from those in
Theorem~\ref{thm:universal} by $-\beta/2$. We work with these shifted frequencies in order to simplify some of the later computations. Clearly, it suffices to prove Theorem~\ref{thm:universal}
for the choice \eqref{eq:seq}.

\subsection{Two auxiliary lemmas}
We start by recording two basic lemmas. The first one is a well-known statement from ergodic theory \cite{Walters1982,EinsiedlerWard2011}. 

\begin{lemma}
\label{lem:two-sided-ergodic}
Let \(\alpha\notin\Q\) and let \(\varphi\in L^1(\T)\). Then, for almost
every \(x\in\T\),
\[
    \frac1{2N+1}\sum_{|k|\leq N}\varphi(x-k\alpha)
    \longrightarrow \int_\T\varphi(y)\,dy.
\]
\end{lemma}

The second lemma is an application of Jensen's formula.

\begin{lemma}
\label{lem:jensen-density-comparison}
Let \(F\neq 0\) be meromorphic on \(\C\), and assume that
\[
 \limsup_{R\to\infty}\frac1R\int_0^{2\pi}\log|F(Re^{i\theta})|\,d\theta
 \leq0.
\]
Then, counting with multiplicity,
\[
 \liminf_{R\to\infty}
 \frac{\#\{\text{zeros of }F\text{ in }|z|\leq R\}}{2R}
 \leq
 \limsup_{R\to\infty}
 \frac{\#\{\text{poles of }F\text{ in }|z|\leq R\}}{2R}.
\]
\end{lemma}

\begin{proof}
Up to multiplying \(F\) by \(z^{-k}\), where \(k\in\Z\) is the order of
\(F\) at the origin, we may assume that \(0<|F(0)|<\infty\) as this
changes the numbers of zeros and poles by at most \(|k|\), and the
integral above by \(O(\log R)\).  Denote by
\(n(t)\) and \(p(t)\) the numbers of zeros and poles of \(F\) in
\(|z|\leq t\), counted with multiplicity.  Whenever the circle
\(|z|=R\) contains no zero or pole of \(F\), Jensen's
formula gives
\[
    \int_0^{R}\frac{n(t)-p(t)}{t}\,dt
    =\frac1{2\pi}\int_0^{2\pi}\log|F(R e^{i\theta})|\,d\theta
     -\log|F(0)|.
\]
If the claimed inequality failed, there would exist
\(\sigma>\rho\geq0\) such that \(n(t)\geq\sigma t\) and
\(p(t)\leq\rho t\) for all sufficiently large \(t\). Hence,
for some constant \(C>0\),
\[
    \int_0^R \frac{n(t)-p(t)}{t}\,dt
    \geq (\sigma-\rho)R-C
\]
for all sufficiently large \(R\). On the other hand, the hypothesis implies that the right-hand side of Jensen's formula is bounded above by \(o(R)\). This is a contradiction, since there are arbitrarily large radii \(R\) for which the circle
\(|z|=R\) contains no zero or pole of \(F\).
\end{proof}

\subsection{A Fourier series identity}
Let $S\subseteq\R$ be measurable with $|S|<1$, and let
$f\in L^1(S)$, extended by zero outside $S$. Further, let $\Lambda = \{ \lambda_n \}$ be given as in \eqref{eq:seq}. We aim to prove Theorem \ref{thm:universal}.

By taking complex conjugates, it suffices to prove that $f=0$ whenever
\begin{equation}\label{eq:fourier-zeros}
 \int_\R f(t)e^{2\pi i\lambda_nt}\,dt=0,
 \quad n\in\Z.
\end{equation}
We start by choosing for $j\in\Z$ a measurable representative of the functions $f_j : [0,1] \to \C$ defined by
\[
 f_j(t):=f(j+t), \quad t\in[0,1).
\]
Then
\begin{equation}\label{eq:layer-sums}
 \sum_{j \in \Z}|\{f_j\ne0\}|=|\{f\ne0\}|\leq|S|<1,
 \quad
 \sum_{j\in \Z}\|f_j\|_1=\|f\|_1,
\end{equation}
and using
$\delta_n=\beta(\{n\alpha\}-1/2)$, it follows from \eqref{eq:fourier-zeros} that

\begin{equation}\label{eq:sampling-identity}
 0 =\int_0^1\sum_{j \in \Z} f_j(t)\,
 e^{2\pi i\beta(t+j)(\{n\alpha\}-1/2)}\,
 e^{2\pi int}\,dt,
 \quad n\in\Z.
\end{equation}

We next consider the Fourier expansion
\begin{equation}\label{eq:sinc-expansion}
 e^{2\pi ia(y-1/2)}
 =\sum_{\ell\in\Z}\frac{\sin(\pi a)}{\pi(a-\ell)}\,
 e^{2\pi i\ell y},\quad y\in(0,1).
\end{equation}
Substituting \eqref{eq:sinc-expansion} formally into
\eqref{eq:sampling-identity}, and then making the change of variables
$x=\{t+\ell\alpha\}$, we get
\begin{equation}\label{eq:sampling-identity2}
  \int_0^1\sum_{j,\ell}
 \frac{\sin\!\bigl(\pi\beta(\{x-\ell\alpha\}+j)\bigr)}
      {\pi\bigl(\beta(\{x-\ell\alpha\}+j)-\ell\bigr)}\,
 f_j(\{x-\ell\alpha\})\,
 e^{2\pi inx}\,dx=0,
 \qquad n\in\Z.  
\end{equation}

This suggests introducing
\[
 M(x):=\sum_{j,\ell}
 \frac{\sin\!\bigl(\pi\beta(\{x-\ell\alpha\}+j)\bigr)}
      {\pi\bigl(\beta(\{x-\ell\alpha\}+j)-\ell\bigr)}
 f_j(\{x-\ell\alpha\}),
\]
as, if interchanging of summation and integration were
justified, \eqref{eq:sampling-identity2} would imply that all Fourier
coefficients of \(M\) (and thus $M$ itself) vanish.

Assuming this were the case, by the structure of the sum above, choosing $\{x-k\alpha\}$ in place of $x$ (with $k\in\Z$) we get new
information at $x$ itself: relabeling also $\ell$ by
$\ell-k$ in the sum, we find that
\[
 0\equiv M(\{x-k\alpha\})=\sum_{j,\ell}
 \frac{\sin\!\bigl(\pi\beta(\{x-\ell\alpha\}+j)\bigr)}
      {\pi\bigl(\beta(\{x-\ell\alpha\}+j)-\ell+k\bigr)}\,
 f_j(\{x-\ell\alpha\}).
\]
The point is that we can shift the denominators by $k\in\Z$, arbitrarily -- yet the sum on the right is always zero. By complex analytic reasons, this will force all coefficients to be zero as well, and in particular $f_j(x)=0$, which will give the result.

Now we make the argument rigorous. To justify the convergence issues, we consider the formal difference
\(M(\{x-k\alpha\})-M(x)\). To be precise, define
\begin{equation}\label{eq:Mk-definition}
\begin{aligned}
 M_k(x):=\sum_{j,\ell}
 &\frac{\sin\!\bigl(\pi\beta(\{x-\ell\alpha\}+j)\bigr)}{\pi}
 f_j(\{x-\ell\alpha\})\\[-1mm]
 &\quad\times\left(
 \frac1{\beta(\{x-\ell\alpha\}+j)-\ell+k}
 -\frac1{\beta(\{x-\ell\alpha\}+j)-\ell}
 \right).
\end{aligned}
\end{equation}
We take the right-hand side as the definition of \(M_k\), without
requiring \(M\) itself to be defined. The subtraction cancels the
non-summable first-order tail, since
\[
 \frac1{a-\ell+k}-\frac1{a-\ell}
 =\frac{-k}{(a-\ell+k)(a-\ell)}
 =O_k(|\ell|^{-2}).
\]
Thus \(M_k\) has substantially better convergence properties than the
formal series \(M\). Moreover, as the next lemma shows, the original
sampling identities imply that every \(M_k\) vanishes.  

\begin{lemma}\label{lem:M-fourier}
For every $k\in\Z$, the series in \eqref{eq:Mk-definition} converges
absolutely in $L^1(0,1)$ and defines a function $M_k$
which vanishes almost everywhere in $(0,1)$.
\end{lemma}

\begin{proof}
For $k=0$ the claim is immediate. Therefore, let $k\in\Z\setminus\{0\}$.
The Fourier coefficients of $y\mapsto e^{2\pi ia(y-1/2)}$ are given by
\[
 \int_0^1e^{2\pi ia(y-1/2)}e^{-2\pi i\ell y}\,dy
 =\frac{\sin(\pi a)}{\pi(a-\ell)}.
\]
Consequently, the values
\begin{equation}\label{eq:difference-coefficient}
 \frac{\sin(\pi a)}{\pi}
 \left(\frac1{a-\ell+k}-\frac1{a-\ell}\right)
 =\frac{\sin(\pi a)}{\pi}\,
 \frac{-k}{(a-\ell+k)(a-\ell)}
\end{equation}
are the Fourier coefficients of
$\bigl(e^{2\pi iky}-1\bigr)e^{2\pi ia(y-1/2)}$.
Since
$|\sin(\pi a)|\leq\pi|a-m|$ for every $m\in\Z$, the absolute value of
\eqref{eq:difference-coefficient} is at most
$|k|/\max(|a-\ell|,|a-\ell+k|)\leq2$, and at most
$2|k|/(\pi|a-\ell|^2)$ when $|a-\ell|\geq2|k|$. We therefore obtain that

\begin{equation}\label{eq:coefficient-bound}
 \sum_{\ell\in\Z}
 \left|\frac{\sin(\pi a)}{\pi}
 \left(\frac1{a-\ell+k}-\frac1{a-\ell}\right)\right|
 \leq C_k
 \quad\text{for every }a\in\R.
\end{equation}

After the change of variables $x=\{t+\ell\alpha\}$, the sum of the $L^1(0,1)$-norms of the terms in \eqref{eq:Mk-definition} is at most $C_k\sum_j\|f_j\|_1<\infty$, by
\eqref{eq:coefficient-bound} and \eqref{eq:layer-sums}. Hence the series converges absolutely in $L^1(0,1)$, and Fubini's theorem allows for termwise integration and rearrangement of the sums.

The $1$-periodic function given by
$\bigl(e^{2\pi iky}-1\bigr)e^{2\pi ia(y-1/2)}$ on $[0,1)$ is continuous,
and its Fourier coefficients \eqref{eq:difference-coefficient} are
absolutely summable by \eqref{eq:coefficient-bound}. Its Fourier
series therefore converges to it at every point.  At $y=\{n\alpha\}$
we obtain
\[
 \sum_{\ell\in\Z}\frac{\sin(\pi a)}{\pi}
 \left(\frac1{a-\ell+k}-\frac1{a-\ell}\right)e^{2\pi i\ell n\alpha}
 =\bigl(e^{2\pi ikn\alpha}-1\bigr)e^{2\pi ia(\{n\alpha\}-1/2)}.
\]

Using this identity with $a=\beta(t+j)$, we compute, for every $n\in\Z$,
\begin{align*}
 \int_0^1 M_k(x)\,e^{2\pi inx}\,dx
 &=\sum_{j,\ell}\int_0^1 f_j(t)e^{2\pi int}
 \frac{\sin(\pi a)}{\pi}
 \Bigl(\frac1{a-\ell+k}-\frac1{a-\ell}\Bigr)
 e^{2\pi i\ell n\alpha}dt\\
 &=\int_0^1\sum_j f_j(t)e^{2\pi int}\sum_{\ell}
 \frac{\sin(\pi a)}{\pi}
 \Bigl(\frac1{a-\ell+k}-\frac1{a-\ell}\Bigr)
 e^{2\pi i\ell n\alpha}dt\\
 &=\bigl(e^{2\pi ikn\alpha}-1\bigr)
 \int_0^1\sum_j f_j(t)\,
 e^{2\pi i\beta(t+j)(\{n\alpha\}-1/2)}\,e^{2\pi int}\,dt=0.
\end{align*}
The last equality is \eqref{eq:sampling-identity}. By uniqueness of
Fourier coefficients in $L^1(0,1)$, we have $M_k=0$ almost everywhere.
\end{proof}

\subsection{A meromorphic function}
We now replace the integer parameter $k$ in \eqref{eq:Mk-definition}
by a complex variable $z$, i.e. we define
\begin{equation}\label{eq:M-definition}
 M_z(x):=\sum_{j,\ell\in\Z}
 r_{\ell,j}\left(\frac1{z-p_{\ell,j}}+\frac1{p_{\ell,j}}\right),
\end{equation}
where
\begin{equation}\label{eq:M-poles}
 p_{\ell,j}=\ell-\beta(\{x-\ell\alpha\}+j),
 \quad
 r_{\ell,j}=\frac{\sin\!\bigl(\pi\beta(\{x-\ell\alpha\}+j)\bigr)}{\pi}\,
 f_j(\{x-\ell\alpha\}).
\end{equation}
At $z=k\in\Z$, the series coincides with \eqref{eq:Mk-definition}. The next lemma establishes the convergence
of the series $M_z$ and describes its poles and zeros.

\begin{lemma}\label{lem:M-meromorphic}
For $z\in\C$, let the series $M_z$ be defined as in \eqref{eq:M-definition}.
For almost every $x\in[0,1)$, it defines a meromorphic function of $z$
with the following properties.
\begin{enumerate}[(i)]
\item $\sum_{j,\ell\in\Z}|r_{\ell,j}|/(1+p_{\ell,j}^2)<\infty$.
\item Its poles are precisely the values $p_{\ell,j}$ for which
$f_j(\{x-\ell\alpha\})\ne0$, and they are real, simple, pairwise
distinct, and disjoint from $\Z$.
\item The residue at $p_{\ell,j}$ is $r_{\ell,j}$.

\item We have $M_k(x)=0$ for every $k\in\Z$. If
$z\mapsto M_z(x)$ is not identically zero, its zeros have lower
density at least one:
\begin{equation}\label{eq:M-integer-zeros}
 \liminf_{R\to\infty}
 \frac{\#\{\text{zeros in }|z|\leq R\}}{2R}\geq1.
\end{equation}
\item The poles have upper density at most $|S|<1$:
\begin{equation}\label{eq:M-pole-density}
 \limsup_{R\to\infty}
 \frac{\#\{\text{poles in }|z|\leq R\}}{2R}\leq|S|<1.
\end{equation}
\end{enumerate}
By (ii), if $z\mapsto M_z(x)$ vanishes identically then
$f_j(x)=0$ for every $j\in\Z$.
\end{lemma}

\begin{proof}
For fixed $j$ and $\ell$, we have that $a(x):=\beta(\{x-\ell\alpha\}+j)$ ranges over
an interval of length $|\beta|<1$, so it is an integer for at most one
$x\in[0,1)$. Ranging over $j$ and $\ell$, these form countably many values
of $x$. Discarding these values, the sine in \eqref{eq:M-poles} never vanishes,
so $r_{\ell,j}\ne0$ precisely when $f_j(\{x-\ell\alpha\})\ne0$, and no
$p_{\ell,j}$ is an integer.

\smallskip\noindent\textbf{Claim 1: Summability.}
We first show that $\sum_{j,\ell}|r_{\ell,j}|/(1+p_{\ell,j}^2)<\infty$
for almost every $x$.  The poles are close to integers:
$p_{\ell,j}=\ell-\beta(\{x-\ell\alpha\}+j)$ lies within $|\beta|<1/2$
of $\ell-\beta j$, hence within $3/2$ of the integer
$q_{\ell,j}:=\ell-\lfloor\beta j\rfloor$, so that
$1+p_{\ell,j}^2\asymp1+q_{\ell,j}^2$.  We now parametrize the pairs
$(j,\ell)$ by $(j,q)$ via $\ell=\lfloor\beta j\rfloor+q$. Since
$|r_{\ell,j}|\leq|f_j(\{x-\ell\alpha\})|/\pi$, it suffices to show
that
\[
 \sum_{j,q}\frac{|f_j(\{x-(\lfloor\beta j\rfloor+q)\alpha\})|}{1+q^2}
 <\infty
\]
for almost every $x$.  For later use, we will in fact show that
\begin{equation}\label{eq:summability}
 \sum_{j,q}
 \frac{|f_j(\{x-(\lfloor\beta j\rfloor+q)\alpha\})|
 +\1_{\{f_j\ne0\}}(\{x-(\lfloor\beta j\rfloor+q)\alpha\})}
 {1+q^2}<\infty
\end{equation}
for almost every $x$.  Since
$\int_0^1|f_j(\{x-m\alpha\})|\,dx=\|f_j\|_1$ and
$\int_0^1\1_{\{f_j\ne0\}}(\{x-m\alpha\})\,dx=|\{f_j\ne0\}|$ for every
$m\in\Z$, by Fubini's theorem the integral over $x\in[0,1)$ of the
left-hand side of \eqref{eq:summability} equals
\[
 \sum_q\frac1{1+q^2}\sum_j\bigl(\|f_j\|_1+|\{f_j\ne0\}|\bigr),
\]
which is finite by \eqref{eq:layer-sums}. Hence the left-hand side of
\eqref{eq:summability} is finite for almost every $x$.

\smallskip\noindent\textbf{Claim 2: Residues.} 
The indicator term in \eqref{eq:summability}, together with
$1+p_{\ell,j}^2\asymp1+q_{\ell,j}^2$, gives
\[
 \sum_{\substack{j,\ell\\r_{\ell,j}\ne0}}
 \frac{1}{1+p_{\ell,j}^2}<\infty.
\]
For each $R>0$, every summand with $|p_{\ell,j}|\le R$
is at least $(1+R^2)^{-1}$. Hence only finitely many
$p_{\ell,j}$ with $r_{\ell,j}\ne0$ lie in any bounded set.

If a compact set $K\subset\C$
contains none of them, then for $z\in K$ and all but finitely many of
them,
\[
 \left|\frac1{z-p_{\ell,j}}+\frac1{p_{\ell,j}}\right|
 =\frac{|z|}{|p_{\ell,j}|\,|z-p_{\ell,j}|}
 \leq\frac{C_K}{1+p_{\ell,j}^2},
\]
so by (i) the series \eqref{eq:M-definition} converges normally on
$K$: $z\mapsto M_z(x)$ is meromorphic, with poles among the
$p_{\ell,j}$ and with $r_{\ell,j}\ne0$.  These are pairwise distinct:
$p_{\ell,j}=p_{\ell',j'}$ gives
\[
 \{x-\ell\alpha\}-\{x-\ell'\alpha\}
 =\frac{\ell-\ell'}{\beta}-(j-j')\in\Q,
\]
while the left-hand side differs from $-(\ell-\ell')\alpha$ by an
integer, $\alpha$ being irrational, $\ell=\ell'$, and then $j=j'$.
Hence, near each $p_{\ell,j}$ with $r_{\ell,j}\ne0$, all terms of
\eqref{eq:M-definition} except for one are holomorphic, and $M_z(x)$ has a
simple pole with residue $r_{\ell,j}$ there.  Together with claim (i), and the fact that $p_{\ell,j}$ is real, we obtain (ii) and (iii).

\smallskip\noindent\textbf{Claim 3: Zeros.}
The series converges at every integer (none of which is a pole). Moreover, at every integer it coincides with \eqref{eq:Mk-definition}, by
Lemma~\ref{lem:M-fourier}. Discarding one more null set, we have $M_k(x)=0$
for every $k\in\Z$.  If $M_z(x)$ is not identically zero, its zeros contain $\Z$, giving
\eqref{eq:M-integer-zeros}.

\smallskip\noindent\textbf{Claim 4: Density of the poles.}
By (ii), the poles are the $p_{\ell,j}$ with $f_j(\{x-\ell\alpha\})\ne0$.
Write $\ell=\lfloor\beta j\rfloor+q$ as in (i): each pole $p_{\ell,j}$
lies within distance $3/2$ of $q_{\ell,j}$, and the number of poles $p_{\ell,j}$
with $q_{\ell,j}=q$ is $\varphi(\{x-q\alpha\})$, where
\[
 \varphi(y):=\sum_j\1_{\{f_j\ne0\}}\bigl(\{y-\lfloor\beta j\rfloor\alpha\}\bigr).
\]
Hence
\[
 \#\{\text{poles in }|z|\leq R\}\leq\sum_{|q|\leq R+2}\varphi(\{x-q\alpha\}).
\]
Now observe that by \eqref{eq:layer-sums} we have
$$
\int_0^1\varphi
=\sum_j\int_0^1\1_{\{f_j\ne0\}}(\{y-\lfloor\beta j\rfloor\alpha\})\,dy
=\sum_j|\{f_j\ne0\}|=|\{f\ne0\}|.
$$
Thus,
$\varphi\in L^1(0,1)$ and Lemma~\ref{lem:two-sided-ergodic} gives,
\[
 \limsup_{R\to\infty}\frac{\#\{\text{poles in }|z|\leq R\}}{2R}
 \leq\limsup_{R\to\infty}\frac1{2R}\sum_{|q|\leq R+2}\varphi(\{x-q\alpha\})
 =|\{f\ne0\}|\leq|S|,
\]
for almost every $x$. This yields \eqref{eq:M-pole-density}.

\smallskip\noindent\textbf{Claim 5: Vanishing.}
If $z\mapsto M_z(x)$ vanishes identically, then it has no poles, so $f_j(\{x-\ell\alpha\})=0$ for every $j$ and $\ell$ by (ii). The case $\ell=0$ gives $f_j(x)=0$ for every $j$. This completes the proof of the statement.
\end{proof}

\subsection{Proof of Theorem~\ref{thm:universal}}
With the help of the previous sections, we are ready to provide a proof of Theorem \ref{thm:universal}

\begin{proof}[Proof of Theorem~\ref{thm:universal}]
First we observe that for every $n\ne m$ we have
\[
 |\lambda_n-\lambda_m|\geq1-|\beta|>\frac12,
\]
so $\Lambda$ is uniformly discrete.  Clearly, $\Lambda$ is also uniformly distributed with $D(\Lambda)=1$.

Let $f$ satisfy \eqref{eq:fourier-zeros}, and construct $M_z(x)$ as
in the previous section. Fix $x$ in a full measure set for which
Lemmas~\ref{lem:M-fourier} and \ref{lem:M-meromorphic} apply, and
suppose for contradiction that $z\mapsto M_z(x)$ is not identically zero.

We first show that
\[
 |M_{Re^{i\theta}}(x)|\leq C_xR\bigl(1+|\sin\theta|^{-1}\bigr),
 \quad R\geq1,
\]
which, by the integrability of $\log(1+|\sin\theta|^{-1})$, gives
\begin{equation}\label{eq:M-log-growth}
 \int_0^{2\pi}\log|M_{Re^{i\theta}}(x)|\,d\theta\leq C_x\log R,
 \quad R\geq2.
\end{equation}
By Lemma~\ref{lem:M-meromorphic}(ii), the poles are locally finite
and avoid zero. Hence their absolute values are bounded below by a
positive constant depending on $x$. If $p$ is a pole and $r$ its
residue, the term of \eqref{eq:M-definition}
corresponding to $p$ satisfies for $z=Re^{i\theta}$ the identity
\[
 \left|r\left(\frac1{z-p}+\frac1p\right)\right|
 =\frac{|r|R}{|p|\,|z-p|}.
\]
For $|p|>2R$ we have $|z-p|\geq|p|-R\geq|p|/2$.  Otherwise, we use
that $|z-p|\geq R|\sin\theta|$ since $p$ is real, and that
$(1+p^2)/|p|\leq C_xR$ by $|p|\leq2R$ and the poles being bounded away
from zero.  Hence
\[
 \left|r\left(\frac1{z-p}+\frac1p\right)\right|
 \leq
 \begin{cases}
  \dfrac{2R\,|r|}{p^2}
  \leq\dfrac{4R\,|r|}{1+p^2}
  & |p|>2R,\\[3mm]
  \dfrac{|r|}{|p|\,|\sin\theta|}
  \leq\dfrac{C_xR\,|r|}{(1+p^2)\,|\sin\theta|}
  & |p|\leq2R.
 \end{cases}
\]
Summing over the poles and using (i) of
Lemma~\ref{lem:M-meromorphic} gives the bound.

By \eqref{eq:M-log-growth}, Lemma~\ref{lem:jensen-density-comparison}
applies to $z\mapsto M_z(x)$, and together with
\eqref{eq:M-integer-zeros} and \eqref{eq:M-pole-density} it gives
\[
 1\leq
 \liminf_{R\to\infty}
 \frac{\#\{\text{zeros in }|z|\leq R\}}{2R}
 \leq
 \limsup_{R\to\infty}
 \frac{\#\{\text{poles in }|z|\leq R\}}{2R}
 \leq|S|<1.
\]
This gives a contradiction and therefore $M_z(x)$ vanishes identically. The last
statement of Lemma~\ref{lem:M-meromorphic} gives $f_j(x)=0$ for every
$j$, and $f=0$ almost everywhere.
\end{proof}

\section{Noncompleteness of asymptotically integer frequencies}
\label{sec:asymptotic-integer}

In this section we prove Theorem~\ref{thm:null}. We start with two reductions. First, we remove the integer frequencies: if we can find a
nonzero \(G\in L^2(\R)\), with arbitrarily small support, and with
\(\hat G(\lambda_n)=0\) whenever \(\lambda_n\notin\Z\), then the function
\(F\) defined by
\[
    \hat F(x)=(1-e^{2\pi ix})\hat G(x)
\]
has \(\hat F(\lambda_n)=0\) for every \(n\in\Z\). Moreover, \(F\) is the
difference of two integer translates of \(G\), so
\(|\supp (F)|\leq2|\supp (G)|\). Thus, we can assume that $\lambda_n\notin\Z$ for every $n$.

Secondly, we remove the frequencies with  \(|\delta_n|\geq1/8\). Since \(\delta_n\to0\), there are only finitely many such frequencies. Choose a nonzero polynomial \(A : \C \to \C \)
which vanishes at the points \(e^{2\pi i\lambda_n}\). If we can find a
nonzero \(G\in L^2(\R)\), with arbitrarily small support, and with
\(\hat G(\lambda_n)=0\) whenever \(|\delta_n|\geq 1/8\), then the function 
\[
    \hat F(x)=A(e^{2\pi ix})\hat G(x),
\]
vanishes at every frequency $\lambda_n$. Moreover, \(F\) is a finite linear combination of a fixed number of integer translates of \(G\). Hence, by taking the support of \(G\) sufficiently small, we obtain \(|\supp (F)|<\varepsilon\) and therefore the desired conclusion.

We have therefore reduced to proving the following statement: \textit{given \(\varepsilon>0\) and \(\lambda_n=n+\delta_n\), with \(\delta_n\to 0\), and assuming additionally that \(0<|\delta_n|<\frac18\), then there is \(F\in L^2(\R)\setminus\{0\}\) with \(|\supp (F)|<\varepsilon\) and \(\hat{F}(\lambda_n) = 0\) for every \(n\in\Z\).
}

\subsection{A preliminary interpolation lemma}
\begin{lemma}
\label{lem:asymp-thin-interpolation}
Let 
\(I\subseteq\Z\) finite, and let
\((v_k)_{k\in I} \subset\C\). Given \(0<\varepsilon\leq1\), there exists
\(f\in L^2(\R)\), with $\supp (f)\subset [0,1]$, $|\supp (f)|\leq \varepsilon$, and $\hat f(\lambda_k)=v_k$ for $k\in I$.
Moreover,
$$\quad
    \norm{f}_2^2\leq\frac{C}{\varepsilon}\sum_{k\in I}|v_k|^2.$$
\end{lemma}

\begin{proof}
If $I=\varnothing$, take $f=0$. We therefore assume in the following that $I\neq\varnothing$.
By Lyapunov's convexity theorem \cite{Liapounoff1940,Barvinok2002} there is a measurable $\Omega\subset(0,1)$ with
$|\Omega|=\varepsilon$ such that
\begin{equation}\label{eq:782tg8uhwboi}
    \int_\Omega e^{2\pi i(\lambda_m-\lambda_n)t}\,dt
    =\varepsilon\int_0^1
    e^{2\pi i(\lambda_m-\lambda_n)t}\,dt,
    \quad n,m\in I.
\end{equation}

Since $\sup_n|\lambda_n-n|\leq1/8<1/4$, Kadec's $1/4$ theorem \cite{Kadets1964,Young2001} yields universal
constants $c,C>0$ such that, for every $(a_k)_{k\in I}\subset\C$,
\begin{equation}\label{eq:asymp-kadec}
    c\sum_{k\in I}|a_k|^2
    \leq
    \int_0^1\left|\sum_{k\in I}a_ke^{2\pi i\lambda_k t}\right|^2dt
    \leq C\sum_{k\in I}|a_k|^2.
\end{equation}
The corresponding Bessel bound gives, for every $h\in L^2(0,1)$
extended by zero outside $(0,1)$, the inequality
\begin{equation}\label{eq:asymp-bessel}
    \sum_{n\in\Z}|\widehat h(\lambda_n)|^2
    \leq C\|h\|_2^2.
\end{equation}
Now choose a compact set $K\subset\Omega$ with
$
    |\Omega\setminus K|<\frac{c\varepsilon}{2\# I}
$
and define the Hermitian matrix $(G_{nm})_{n,m \in I}$ via
\[
    G_{nm}:=\int_K e^{2\pi i(\lambda_m-\lambda_n)t}\,dt,
    \quad n,m\in I.
\]
It is well-known that $G$ is invertible.
Set $a=G^{-1}v$ and define
\[
    f(t):=\1_K(t)\sum_{k\in I}a_ke^{2\pi i\lambda_k t}.
\]
The definition of $f$ gives $|\supp(f)|\leq|K|\leq\varepsilon$, and
\[
    \widehat f(\lambda_n)=(Ga)_n=v_n.
    \quad n\in I,
\]
Finally, we have
\[
    \|f\|_2^2=a^*Ga=v^*G^{-1}v
    \leq\frac{2}{c\varepsilon}\sum_{k\in I}|v_k|^2,
\]
which proves the required estimate after adjusting the universal constant.
\end{proof}

\subsection{Construction of the counterexample}

The proof is based on an iterative correction procedure. We divide the frequencies into dyadic blocks \(B_j\), according to the size of \(|\delta_n|\), and eliminate one block at a time. At stage \(j\), we add a correction \(g_j\) which creates zeros on \(B_j\) while preserving all zeros obtained at earlier stages. A direct application of the interpolation lemma would not give sufficient control over the effect of \(g_j\) on later blocks. We therefore take \(g_j\) to be the difference of two translates of an interpolating function \(h_j\), so that
\[
    \widehat g_j(x)
    =\bigl(e^{2\pi i2^{j-1}x}-1\bigr)\widehat h_j(x).
\]
The multiplier is bounded away from zero on \(B_j\), allowing us to cancel the current Fourier values there. We also arrange that \(\widehat g_j\) vanishes on the nearby blocks, while on much later blocks the multiplier is small because \(\delta_n\to0\). This decay
makes the corrections square summable.

We begin with a nonzero smooth function whose support has measure less than \(\varepsilon/2\) and whose Fourier transform vanishes on finitely many initial blocks. We choose the subsequent corrections \(g_j\) with pairwise disjoint supports, also disjoint from the initial support, and require
\[
    |\supp (g_j)|\leq\varepsilon_j,
    \quad
    \sum_{j\geq J}\varepsilon_j<\frac{\varepsilon}{2},
\]
where \(J\) is the first correction stage. Thus the union of the
initial support and all correction supports has measure less than
\(\varepsilon\). The resulting \(L^2\)-limit is nonzero and its
Fourier transform vanishes at every \(\lambda_n\).

\begin{proof}[Proof of Theorem~\ref{thm:null}]
Fix \(\varepsilon>0\). In view of the preceding reductions, we may
assume that
\[
    0<|\delta_n|<\frac18,\quad n\in\Z.
\]
We divide the construction into three steps.

\medskip
\noindent\textbf{Step 1: Adding zeros on a fixed dyadic block.}
For \(j\geq3\), let
\[
    B_j=\{n\in\Z:2^{-j-1}\leq|\delta_n|<2^{-j}\}.
\]
Each \(B_j\) is finite, since \(\delta_n\to0\), and the blocks
\((B_j)_{j\geq3}\) form a partition of \(\Z\).

For \(j\geq3\), define
\[
    \varepsilon_j
    =\frac{\varepsilon}{4}\,
    \frac{
        2^{-j}+\displaystyle\max_{n\in B_j}\frac1{(1+|n|)^2}
    }{
        \displaystyle\sum_{r\geq3}2^{-r}
        +\displaystyle\sum_{n\in\Z}\frac1{(1+|n|)^2}
    },
\]
where the maximum is understood to be zero when \(B_j=\varnothing\). These numbers are chosen so that, for some constant
\(C_\varepsilon>0\),
\begin{equation}\label{eq:asymp-support-budget}
    \sum_{j\geq3}\varepsilon_j<\frac{\varepsilon}{2},
    \quad
    \frac1{\varepsilon_j}\leq C_\varepsilon 2^j,
    \quad
    \frac1{\varepsilon_j}
    \leq C_\varepsilon(1+|n|)^2,
    \quad n\in B_j.
\end{equation}

Let \(j\geq3\), and suppose that \(F\in L^2(\R)\) has bounded support and satisfies
\[
    \widehat F(\lambda_n)=0,
    \quad n\in B_3\cup\cdots\cup B_{j-1}.
\]
We construct a correction \(g_j\) which adds zeros on \(B_j\), preserves all the preceding zeros, and does not affect the blocks \(B_{j+1},\ldots,B_{4j}\). If \(B_j=\varnothing\), we simply take \(g_j=0\).

For \(n\in B_j\), we have
\[
    \frac14\leq2^{j-1}|\delta_n|<\frac12.
\]
Since \(\lambda_n=n+\delta_n\), it follows that
\begin{equation}\label{eq:asymp-dyadic-multiplier}
\begin{aligned}
    c
    &\leq
    \left|e^{2\pi i2^{j-1}\lambda_n}-1\right|
    \leq C,
    && n\in B_j,\\
    \left|e^{2\pi i2^{j-1}\lambda_n}-1\right|
    &\leq C2^j|\delta_n|,
    && n\in\Z.
\end{aligned}
\end{equation}
The first estimate allows us to divide by the multiplier on \(B_j\).
The second will control the effect of the correction on later blocks.

Choose an integer \(M_j\) sufficiently large that both intervals
\[
    [M_j,M_j+1]
    \quad\text{and}\quad
    [M_j-2^{j-1},M_j-2^{j-1}+1]
\]
are disjoint from \(\supp (F)\). Since \(j\ge3\), these intervals
are also disjoint from each other.
Applying Lemma~\ref{lem:asymp-thin-interpolation}, followed by a
translation to \([M_j,M_j+1]\), we obtain a function \(h\), supported
on a set of measure at most \(\varepsilon_j/2\), such that
\[
    \widehat h(\lambda_n)=
    \begin{cases}
        -\widehat F(\lambda_n)/
        \bigl(e^{2\pi i2^{j-1}\lambda_n}-1\bigr),
        &n\in B_j,\\[1mm]
        0,
        &n\in(B_3\cup\cdots\cup B_{4j})\setminus B_j,
    \end{cases}
\]
and
\[
    \norm{h}_2^2
    \leq\frac{C}{\varepsilon_j}
    \sum_{n\in B_j}|\widehat F(\lambda_n)|^2.
\]
Here the prescribed values are modified by the appropriate unimodular
factors before translating, so that the displayed identities hold for
the translated function.

Now consider
\[
    g_j(t)=h(t+2^{j-1})-h(t),
\]
which satisfies
\[
    \widehat g_j(x)
    =\bigl(e^{2\pi i2^{j-1}x}-1\bigr)\widehat h(x).
\]
Consequently, \(\widehat{F+g_j}\) vanishes on
\(B_3\cup\cdots\cup B_j\), while \(\widehat g_j\) vanishes on
\(B_{j+1}\cup\cdots\cup B_{4j}\). Since the two translates of \(h\)
have disjoint supports,
\begin{equation}\label{eq:asymp-one-block}
    |\supp (g_j)|\leq\varepsilon_j,
    \quad
    \norm{g_j}_2^2
    \leq\frac{C}{\varepsilon_j}
    \sum_{n\in B_j}|\widehat F(\lambda_n)|^2.
\end{equation}

It remains to estimate the effect of \(g_j\) on the more distant
blocks. If \(k>4j\) and \(n\in B_k\), then
\eqref{eq:asymp-support-budget} and
\eqref{eq:asymp-dyadic-multiplier} give
\[
    \frac{
        |e^{2\pi i2^{j-1}\lambda_n}-1|^2
    }{\varepsilon_k}
    \leq C_\varepsilon4^{-j}.
\]
Using the Bessel estimate \eqref{eq:asymp-bessel}, after translating
\(h\) back to an interval of length one, we obtain
\begin{equation}\label{eq:asymp-far-blocks}
\begin{aligned}
    \sum_{k>4j}\frac1{\varepsilon_k}
    \sum_{n\in B_k}|\widehat g_j(\lambda_n)|^2
    &=
    \sum_{k>4j}\frac1{\varepsilon_k}
    \sum_{n\in B_k}
    |e^{2\pi i2^{j-1}\lambda_n}-1|^2
    |\widehat h(\lambda_n)|^2\\
    &\leq
    C_\varepsilon4^{-j}
    \sum_{k>4j}\sum_{n\in B_k}
    |\widehat h(\lambda_n)|^2\\
    &\leq
    \frac{C_\varepsilon4^{-j}}{\varepsilon_j}
    \sum_{n\in B_j}|\widehat F(\lambda_n)|^2.
\end{aligned}
\end{equation}
This ensures that the corrections are square summable.
\medskip

\noindent\textbf{Step 2: Corrections on every block.}
Choose \(J\geq3\) so large that
\[
    C_\varepsilon^2\sum_{j\geq J}4^{-j}\leq\frac14,
\]
where \(C_\varepsilon\) is large enough for the estimates above. Take
a nonzero function \(\psi\in C_c^\infty(\R)\) with
\[
    |\supp(\psi)|<\frac{\varepsilon}{2},
\]
and define
\[
    F_{J-1}
    =
    \prod_{n\in B_3\cup\cdots\cup B_{J-1}}
    \left(\frac1{2\pi i}\frac{d}{dt}-\lambda_n\right)\psi.
\]
Then \(F_{J-1}\neq0\), its support is contained in \(\supp(\psi)\), and
its Fourier transform vanishes on
\(B_3\cup\cdots\cup B_{J-1}\). Because \(\psi\) is smooth, these
finitely many zeros can be introduced by differentiation without
enlarging its support. The later corrections are only \(L^2\), which
is why we use the interpolation construction from Step~1 for the
remaining blocks.

Since \(\widehat F_{J-1}\) is rapidly decreasing, the last estimate in
\eqref{eq:asymp-support-budget} gives
\[
    \sum_{j\geq J}\frac1{\varepsilon_j}
    \sum_{n\in B_j}
    |\widehat F_{J-1}(\lambda_n)|^2<\infty.
\]
Multiplying \(F_{J-1}\) by a nonzero constant, we may assume that
\begin{equation}\label{eq:asymp-initial-energy}
    \sum_{j\geq J}\frac1{\varepsilon_j}
    \sum_{n\in B_j}
    |\widehat F_{J-1}(\lambda_n)|^2\leq1.
\end{equation}

We now proceed through the blocks \(B_J,B_{J+1},\ldots\). At stage
\(j\), apply Step~1 to \(F_{j-1}\), place the correction \(g_j\) on a
part of the real line disjoint from all the preceding supports, and set
\[
    F_j=F_{j-1}+g_j.
\]
Thus each step adds one new block of zeros, preserves all preceding
zeros, and increases the measure of the support by at most
\(\varepsilon_j\).

We next verify that the corrections are square summable. For \(L\geq J\),
set
\[
    A_L=
    \left(
        \sum_{j=J}^{L}\frac1{\varepsilon_j}
        \sum_{n\in B_j}
        |\widehat F_{j-1}(\lambda_n)|^2
    \right)^{1/2}.
\]
For \(t<j\), the correction \(g_t\) vanishes on \(B_j\) whenever
\(j\leq4t\). Thus only corrections with \(j>4t\) contribute to the
values on \(B_j\). Minkowski's inequality,
\eqref{eq:asymp-initial-energy}, and
\eqref{eq:asymp-far-blocks} therefore give
\[
    A_L
    \leq
    1+
    \sum_{t=J}^{L-1}
    \frac{C_\varepsilon2^{-t}}{\sqrt{\varepsilon_t}}
    \left(
        \sum_{n\in B_t}
        |\widehat F_{t-1}(\lambda_n)|^2
    \right)^{1/2}.
\]
By Cauchy-Schwarz and the choice of \(J\),
\[
    A_L
    \leq
    1+
    C_\varepsilon
    \left(\sum_{t\geq J}4^{-t}\right)^{1/2}A_L
    \leq1+\frac12A_L.
\]
Hence \(A_L\leq2\), uniformly in \(L\), and consequently
\[
    \sum_{j\geq J}\frac1{\varepsilon_j}
    \sum_{n\in B_j}
    |\widehat F_{j-1}(\lambda_n)|^2
    \leq4.
\]
It now follows from \eqref{eq:asymp-one-block} that
\[
    \sum_{j\geq J}\norm{g_j}_2^2<\infty.
\]

\medskip
\noindent\textbf{Step 3: Passing to the limit.}
The supports of the corrections were chosen pairwise disjoint and
disjoint from \(\supp (F_{J-1})\). Therefore, we have that
\[
    F=F_{J-1}+\sum_{j\geq J}g_j
\]
converges in \(L^2(\R)\). Since \(F\) agrees with
the nonzero function \(F_{J-1}\) on \(\supp (F_{J-1})\), it follows that $F \neq 0$.

In addition, we observe that
\[
    \sum_{j\geq J}\norm{g_j}_1
    \leq
    \left(\sum_{j\geq J}|\supp (g_j)|\right)^{1/2}
    \left(\sum_{j\geq J}\norm{g_j}_2^2\right)^{1/2}
    <\infty.
\]
It follows that \(\widehat F_j\) converges uniformly to \(\widehat F\).
At stage \(j\), the Fourier transform of \(F_j\) vanishes on
\(B_3\cup\cdots\cup B_j\), and every subsequent correction preserves
these zeros. Since the blocks \(B_j\), where \(j\geq3\), cover \(\Z\), we
conclude that for all $n \in \Z$ it holds that $\widehat F(\lambda_n)=0$.

Finally,
\[
    |\supp (F)|
    \leq
    |\supp (F_{J-1})|+\sum_{j\geq J}\varepsilon_j
    <\varepsilon.
\]
Taking \(S=\supp (F)\), the nonzero function \(F\in L^2(S)\) is
orthogonal to \(E(\Lambda)\), and the theorem follows.
\end{proof}

\section{Universal completeness of integer frequencies on subsets of the unit interval}\label{sec:integer}
This section proves Theorem~\ref{thm:integer}. Fix an $\alpha\in\R\setminus\Q$ and $v\in[0,1]$. Recall that we define
\begin{equation}\label{eq:construction}
 \Lambda_v=\{n\in\Z:\{n\alpha\}\in[1-v,1)\} .
\end{equation}
For $v=0$, Theorem \ref{thm:integer} is trivial. For $v=1$ we have $\Lambda_v=\Z$, and
Theorem \ref{thm:integer} reduces to the classical fact that a function in $L^1(0,1)$ whose
Fourier coefficients all vanish is zero almost everywhere.  Hence we
assume $v\in(0,1)$ from now on.

We will use the following consequence of Jensen's formula, see
\cite[Lectures~2-3]{Levin1996}.

\begin{proposition}\label{prop:zero-density}
    Let $I\subseteq[0,1]$ be an interval and let $g\in L^1(I)$, extended by
zero to $[0,1]$.  If
\[
 \liminf_{N\to\infty}
 \frac{\#\{k\in\Z:|k|\leq N,\ \widehat g(k)=0\}}{2N+1}>|I|,
\]
then $g=0$ almost everywhere
\end{proposition}

\subsection{Density of $\Lambda_v$}
We need the following well-known preliminary lemma:

\begin{lemma}
\label{lem:uniform-interval}
Let $\alpha\notin\Q$ and let $J\subseteq[0,1]$ be an interval.  Then
\[
 \lim_{N\to\infty}\ \sup_{x\in[0,1]}
 \left|
  \frac1N\sum_{k=0}^{N-1}\1_J(\{x+k\alpha\})-|J|
 \right|=0.
\]
\end{lemma}

We can now show that $\Lambda_v$ has density $v$.

\begin{proposition}\label{prop:density}
We have $D(\Lambda_v)=v$.
\end{proposition}
\begin{proof}[Proof of Proposition~\ref{prop:density}]
Put $J=[1-v,1)\subseteq[0,1]$.  For $M\in\Z$ and
$N\geq1$,
\begin{align*}
 \#\bigl(\Lambda_v\cap\{M,M+1,\ldots,M+N-1\}\bigr)
 &=\sum_{k=0}^{N-1}\1_J(\{(M+k)\alpha\})\\
 &=\sum_{k=0}^{N-1}\1_J(\{\,\{M\alpha\}+k\alpha\,\}).
\end{align*}
Applying Lemma~\ref{lem:uniform-interval} with
$x=\{M\alpha\}$, and using $|J|=v$, we obtain
\begin{equation}\label{eq:integer-block-density}
 \sup_{M\in\Z}
 \left|
  \frac{\#(\Lambda_v\cap[M,M+N))}{N}-v
 \right| \to 0,
\end{equation}
from which $D(\Lambda_v)=v$ follows directly.
\end{proof}

\subsection{Reduction to a band-limited function and conclusion}
We will deduce completeness in Theorem~\ref{thm:integer}
from the classical Theorem~\ref{prop:zero-density}, via a reduction. Fix a measurable $S\subseteq[0,1]$ with $|S|<v\in(0,1)$,
not necessarily an interval anymore.

If $E(\Lambda_v)$ were not complete in $L^p(S)$ for some
$p\in[1,\infty)$, i.e.
$\overline{\lspan}\,E(\Lambda_v)\neq L^p(S)$, then by Hahn-Banach there would be a nonzero
$f\in L^{p'}(S)$, $1/p+1/p'=1$, with
\[
 \int_S f(x)e^{-2\pi i x\lambda}\, dx=0\quad\mbox{for all}\quad 
 \lambda\in\Lambda_v .
\]
Since $|S|<1$ we have $f\in L^1(S)$ as well. Thus to get a contradiction it suffices to show that every $f\in L^1(S)$ with
$\widehat f(\lambda)=0$ for all $\lambda\in\Lambda_v$ implies that $f = 0$ almost everywhere. We fix (a representative of) such an $f$ for the remainder of the section.

The strategy is as follows.  Fix some $x\in[0,1]$. We want to see
that $f(x)=0$. Consider the sequence
\begin{equation}\label{eq:1892gtoqgw}
     \widetilde c_k=f(\{x-k\alpha\}),\quad k\in\Z .
\end{equation}
By the (formal, for now) Fourier inversion formula for $f$,
\[
 \widetilde c_k=\sum_{n\in\Z}\widehat f(n)e^{2\pi i n(x-k\alpha)}
    =\sum_{n\in\Z}\widehat f(n)e^{2\pi i nx}\,e^{-2\pi i k\{n\alpha\}}.
\]
Formally, the $\widetilde c_k$ are the Fourier coefficients of the measure
\[
 \mu_x(t):=\sum_{n\in\Z}\widehat f(n)e^{2\pi i nx}\,\delta_{\{n\alpha\}}(t),
\]
since
\[
 \widetilde c_k=\int_0^1e^{-2\pi i kt}\, d\mu_x(t)=\widehat \mu_x(k).
\]
The key point is that, since $\widehat f(n)=0$ for every $n\in\Lambda_v$, only the $n$ with
$\{n\alpha\}\in[0,1-v)$ contribute to the sum defining $\mu_x$, thus
\begin{equation}\label{eq:032htqbgra}
    \supp(\mu_x)\subseteq [0,1-v]=:I.
\end{equation}
Moreover, from \eqref{eq:1892gtoqgw} and the fact that $\alpha$ is
irrational, the standard pointwise ergodic theorem
(Lemma~\ref{lem:two-sided-ergodic}) shows (for an a.e. choice of $x$) that
$\widetilde c_k=\widehat \mu_x(k)$ is zero as frequently as $f$ is, i.e.
\begin{equation}\label{eq:12t82gh98}
 \liminf_{N\to\infty}
 \frac{\#\{k\in\Z:|k|\leq N,\ \widehat \mu_x(k)=0\}}{2N+1}\geq |[0,1]\setminus S|>1-v=|I| .
\end{equation}
If $\mu_x(t)$ were an actual function in $L^1$, \eqref{eq:032htqbgra}
and \eqref{eq:12t82gh98} would show precisely that we can apply
Proposition~\ref{prop:zero-density} to $\mu_x$.  Then
$\widehat \mu_x(k)=0$ for {\it every} $k\in\Z$, thus in particular
\[
 f(x)=\widetilde c_0=\widehat \mu_x(0)=0,
\]
concluding the result.

Since $\mu_x$ is not a function, we need to consider a regularisation $c_k:=w_k \widetilde c_k$ instead. We choose $w_k=\widehat W(k)$, where $W=\1_{[0,\varepsilon/2]}*\1_{[0,\varepsilon/2]}$: the $w_k$ are then summable, and $\supp (W)\subseteq [0,\varepsilon]$. This replaces $\mu_x$ by the continuous function $g_x:=W*\mu_x$. Moreover, we have the following properties:
\begin{enumerate}
    \item For $\varepsilon>0$ irrational, every
$w_k$ is nonzero. Thus $\widehat g_x(k)=c_k=0$  if and only if $\widetilde c_k=0$.
\item $\supp (g_x)=\supp (W*\mu_x)\subseteq [0,\varepsilon]+\supp (\mu_x)\subseteq[0, 1-v+\varepsilon]=:I_\varepsilon$.
\end{enumerate}
Choosing
$\varepsilon$ small enough so that $|[0,1]\setminus S|>1-v+\varepsilon=|I_\varepsilon|$, we observe that the previous argument still goes through.
The remainder of the proof makes these observations rigorous. We start by describing properties of $W$.

\begin{lemma}\label{lem:window}
Fix an irrational $\varepsilon\in(0,v)$. Then $W$ is continuous, $\supp (W)\subseteq [0,\varepsilon]$, and $w_k\neq0$ for every $k\in\Z$. Explicitly,
\[
 w_0=\frac{\varepsilon^2}4,
 \quad
 w_k=\Bigl(\frac{1-e^{-\pi i k\varepsilon}}{2\pi i k}\Bigr)^{\!2}
 \ \ (k\neq0),
 \quad
 \sum_{k\in\Z}|w_k|=\frac\varepsilon2 .
\]
\end{lemma}

\begin{proof}
The Fourier coefficients of a convolution multiply, so $w_k=a_k^2$ with
$a_k=\int_0^{\varepsilon/2}e^{-2\pi i kt}\, dt$, and
the displayed values follow by computing the integral. Parseval's
identity gives
$\sum_k|w_k|=\varepsilon/2$.  If $w_k=0$ then $k\neq0$ and $e^{-\pi i k\varepsilon}=1$, contradicting the irrationality of
$\varepsilon$.  Finally, $W$ is continuous, being a convolution of two
$L^2$ functions, and similarly $\supp (W)\subseteq [0,\varepsilon/2]+[0,\varepsilon/2]=[0,\varepsilon]$.
\end{proof}

Recall that we fixed a representative of  $f\in L^1(0,1)$, satisfying $\widehat f(\lambda)=0$ for every $\lambda\in\Lambda_v$, at the beginning of the section. We next construct an auxiliary function $g_x$ that will be useful later.

\begin{lemma}
\label{lem:broadening}
For $x\in[0,1]$ put $c^{(x)}_k=w_kf(\{x-k\alpha\})$.  Then, for almost
every $x\in[0,1]$, we have $\{c^{(x)}_k\}_{k\in\Z}\in\ell^1(\Z)$ and
there is $g_x\in C([0,1])$ with
\[
 \widehat{g_x}(k)=c^{(x)}_k, \quad  k\in\Z,
 \quad
 \supp (g_x)\subseteq I_\varepsilon:=[0,1-v+\varepsilon] .
\]
\end{lemma}

\begin{proof}
Put
$c^{(x)}(f)=\bigl(w_kf(\{x-k\alpha\})\bigr)_{k\in\Z}=c^{(x)}$. Fubini and Lemma~\ref{lem:window} give
\begin{equation}\label{eq:c-tonelli}
 \int_0^1\norm{c^{(x)}(f)}_{\ell^1(\Z)}\, dx
 =\sum_{k\in\Z}|w_k|\,\norm{f}_{L^1}
 =\frac\varepsilon2\norm{f}_{L^1} ,
\end{equation}
hence $c^{(x)}\in\ell^1(\Z)$ for a.e. $x$. Therefore,
$g_x(t)=\sum_{k\in\Z}c^{(x)}_ke^{2\pi i kt}$ defines a continuous
function with $\widehat{g_x}(k)=c^{(x)}_k$. To see that $\supp (g_x)\subseteq I_\varepsilon$ we argue by approximation.

To do so, let $f_j$ be the Fej\'er means of $f$ with
$\sum_j\norm{f_j-f}_{L^1}<\infty$ (possible since the Fej\'er
means converge to $f$ in $L^1([0,1])$
\cite[Chapter~I]{Katznelson2004}).
Each $f_j$ is a trigonometric polynomial whose Fourier coefficients are
multiples of those of $f$, hence they vanish on $\Lambda_v$ as well.
Applying \eqref{eq:c-tonelli} to $f_j-f$ and summing over $j$ gives
$c^{(x)}(f_j)\to c^{(x)}$ in $\ell^1(\Z)$ for a.e. $x$, hence (defining $g_{x,j}$ via $\widehat{g_{x,j}}(k)=c^{(x)}_k(f_j)$) we find 
$\norm{g_{x,j}-g_x}_\infty\leq\norm{c^{(x)}(f_j)-c^{(x)}}_{\ell^1(\Z)}
\to0$.

We have then reduced to showing that
$\supp (g_{x})\subseteq I_\varepsilon$ under the additional assumption that $f$ is a trigonometric
polynomial. In this case, the formal computation in the summary before is legitimate, and $\mu_x$ is a finite sum of point
masses. Then $g_x=W*\mu_x$ is the continuous function
\[
 \sum_{n\notin\Lambda_v}\widehat f(n)e^{2\pi i nx}\,W(\,\cdot-\{n\alpha\}) ,
\]
a finite sum of translates of $W$, each supported in
$[\{n\alpha\},\{n\alpha\}+\varepsilon]\subseteq I_\varepsilon$, thus
$\supp (g_x)\subseteq I_\varepsilon$ indeed.
\end{proof}

We can now prove the theorem.

\begin{proof}[Proof of Theorem~\ref{thm:integer}]
The density claim is Proposition~\ref{prop:density} and we are left proving completeness. As described at the beginning of the present section, it suffices to show that the function $f$ with $f\in L^1(S)$ and $\widehat f(\lambda)=0$ for all
$\lambda\in\Lambda_v$ vanishes almost everywhere. Fix
$\varepsilon\in(0,v-|S|)$ irrational, so that
$|I_\varepsilon|=1-v+\varepsilon<1-|S|$.

For almost every choice of $x\in[0,1]$, the conclusions of
Lemma~\ref{lem:broadening} and Lemma~\ref{lem:two-sided-ergodic}
(applied to $\varphi=\1_{[0,1]\setminus S}$) hold.  We will show that $f(x)=0$ and conclude the proof.
Let $g_x$ be the
function from Lemma~\ref{lem:broadening}: $g_x\in C([0,1])$,
$\supp (g_x)\subseteq I_\varepsilon$, and
$\widehat g_x(k)=w_kf(\{x-k\alpha\})$ for $k\in\Z$.  Since $\supp (f) \subseteq S$, we have
$\widehat g_x(k)=0$ whenever $\{x-k\alpha\}\in[0,1]\setminus S$, and
Lemma~\ref{lem:two-sided-ergodic} gives
\[
 \liminf_{N\to\infty}
 \frac{\#\{k\in\Z:|k|\leq N,\ \widehat g_x(k)=0\}}{2N+1}
 \geq 1-|S|>|I_\varepsilon| .
\]
Proposition~\ref{prop:zero-density} now yields $g_x\equiv 0$, thus
$f(x)=w_0^{-1}\widehat g_x(0)=0$ as desired.
\end{proof}

\section{Periodic weak gaps}
\label{sec:sobolev-sharpness}

Fix a measurable set \(A\subseteq[0,1]\) with \(0<|A|<1\), and set
\(S=A+\Z\). Given
\(f=\ft F\) write
\(\norm{f}_{\alpha,S}^2=\int_S(1+|t|^{2\alpha})|F(t)|^2\,dt\).

As discussed in the introduction, to obtain the missing part of Theorem \ref{thm:sobolev} we show that for any \(0\leq\alpha\leq1/2\) and $S$ as above, no uniformly discrete set $\Lambda$ is a uniqueness set for \(PW^{(\alpha)}_S\cap C(\R)\). 

\subsection{Finitely many points}
We start with the case where $\Lambda$ is finite. The general case will follow by an iterative construction. Clearly, for finite $\Lambda$ we have that $\Lambda$ is not a uniqueness set for $PW^{(\alpha)}_S\cap C(\R)$. We require a quantitative version of this, which is the content of the next proposition.

\begin{proposition}
\label{prop:sobolev-peak}
Let \(0\leq\alpha\leq1/2\). Then no finite set is a uniqueness set for
\(PW^{(\alpha)}_S\cap C(\R)\). More precisely, given a finite set
\(\Lambda\subseteq\R\), a compact set \(K\subseteq\R\),
\(y_0\notin K\cup\Lambda\), and \(\varepsilon>0\), there exists
\(f\in PW^{(\alpha)}_S\cap C(\R)\) with
\[
    f(y_0)=1,
    \quad f|_\Lambda=0,
    \quad \sup_{x\in K}|f(x)|<\varepsilon,
    \quad \norm{f}_{\alpha,S}<\varepsilon.
\]
\end{proposition}

In order to prove it, we first obtain some \(f\in L^2\), with no control on \(\sup_{x\in K}|f(x)|\) or \(\norm{f}_{\alpha,S}\):

\begin{lemma}\label{lem:sobolev-moment}
Let \(\Lambda\subseteq\R\) be finite and \(y_0\notin\Lambda\). There exists
\(f\in PW_A\) such that
\[
    f(y_0)=1,
    \quad f|_\Lambda=0.
\]
\end{lemma}

\begin{proof}
Write \(\{y_0\}\cup\Lambda=\{d_0,\ldots,d_m\}\), where \(d_0=y_0\), and define
\[
    T:L^2(A)\longrightarrow\C^{m+1},
    \quad
    TF=\left(\int_AF(t)e^{-2\pi i d_jt}\,dt\right)_{j=0}^m.
\]
We show that \(T\) is onto. If \(\operatorname{im} T\neq\C^{m+1}\), there
exists \((c_0,\ldots,c_m)\in\C^{m+1}\setminus\{0\}\) such that the linear
functional \((z_0,\ldots,z_m)\mapsto\sum_{j=0}^m c_jz_j\) vanishes on
\(\operatorname{im} T\). Hence, for every \(F\in L^2(A)\),
\[
    0=\sum_{j=0}^m c_j\int_AF(t)e^{-2\pi i d_jt}\,dt
     =\int_AF(t)H(t)\,dt,
    \quad H(t):=\sum_{j=0}^m c_je^{-2\pi i d_jt}.
\]
Since \(\overline H|_A\in L^2(A)\), taking \(F=\overline H|_A\) gives
\(H=0\) a.e. on \(A\).

We claim that then \(c_0=\cdots=c_m=0\), obtaining a contradiction. Consider
the analytic extension \(H(z)=\sum_{j=0}^m c_j e^{-2\pi i d_jz}\).
By the
assumptions, its zero set has an accumulation point, and thus \(H\equiv0\).
Differentiating at \(0\) for orders \(n=0,1,\ldots,m\), we obtain
\[
    0=H^{(n)}(0)=(-2\pi i)^n\sum_{j=0}^m c_jd_j^n.
\]
Thus
\[
    \begin{pmatrix}
        1&1&\cdots&1\\
        d_0&d_1&\cdots&d_m\\
        \vdots&\vdots&&\vdots\\
        d_0^m&d_1^m&\cdots&d_m^m
    \end{pmatrix}
    \begin{pmatrix}c_0\\c_1\\\vdots\\c_m\end{pmatrix}=0.
\]
Its determinant is \(\prod_{0\leq j<k\leq m}(d_k-d_j)\), which is
nonzero because the \(d_j\)'s are distinct. Hence
\(c_0=\cdots=c_m=0\), a contradiction, and therefore \(T\) is onto.
Choosing \(F\) with \(TF=(1,0,\ldots,0)\) and setting \(f=\ft F\) proves
the lemma.
\end{proof}

We can now give a proof of the proposition above.

\begin{proof}[Proof of Proposition~\ref{prop:sobolev-peak}]
By translation invariance, we may assume that \(y_0=0\). By Lemma~\ref{lem:sobolev-moment}, there exists \(p\in PW_A\) such that
\[
    p(0)=1,
    \quad p|_{\Lambda\cup(K\cap\Z)}=0.
\]
Since \(|A|<\infty\), we have \(\ift p\in L^1(A)\), and hence \(p\) is
continuous. For \(j\geq0\), let
\[
    I_j=\{2^j,\ldots,2^{j+1}-1\},\quad
    R_j(x)=2^{-j}\sum_{k\in I_j}e^{-2\pi ikx},\quad
    f_j(x)=p(x)R_j(x).
\]
Note that \(R_j(0)=1\) and \(|R_j|\leq1\), thus \(f_j(0)=1\) and \(f_j|_\Lambda=0\), just as for $p$. Moreover, the
geometric-sum formula applied to \(R_j\) gives
\begin{equation}\label{eq:est}
    |R_j(x)|\leq 2^{-j}\frac{2}{|1-e^{-2\pi ix}|},
    \quad x\notin\Z.
\end{equation}
It therefore follows that \(R_j\to0\) uniformly on a set which has positive distance to \(\Z\). Since \(p\) vanishes on
\(K\cap\Z\), continuity and compactness give
\(\sup_{x\in K}|f_j(x)|\to0\), since near \(K\cap\Z\) we can use the estimate \(|R_j|\leq1\),
and away from this set we use the displayed estimate in \eqref{eq:est}.

On the Fourier side, we have
\[
    \ift f_j
    =(\ift p)*\left(2^{-j}\sum_{k\in I_j}\delta_k\right)
    =2^{-j}\sum_{k\in I_j}(\ift p)(\,\cdot-k).
\]
Since \(\supp(\ift p)\subseteq A\), we obtain
\[
    \supp(\ift f_j)\subseteq\bigcup_{k\in I_j}(A+k)\subseteq S,
\]
and \(f_j\in PW^{(\alpha)}_S\cap C(\R)\). Since the translates \(A+k\) are
disjoint up to null sets, it follows that
\begin{align*}
    \norm{f_j}_{\alpha,S}^2
    &=2^{-2j}\sum_{k\in I_j}
      \int_A(1+|t+k|^{2\alpha})|(\ift p)(t)|^2\,dt \\
    &\leq C_\alpha 2^{(2\alpha-1)j}\norm{\ift p}_{L^2(A)}^2.
\end{align*}

If \(\alpha<1/2\), both \(\sup_K|f_j|\) and
\(\norm{f_j}_{\alpha,S}\) tend to zero. Thus one sufficiently large \(j\)
works, and we take \(f=f_j\).

At the critical value \(\alpha=1/2\), a large \(j\) still makes \(f_j\) small on
\(K\), but the last estimate gives only the uniform bound
\(\norm{f_j}_{1/2,S}^2\leq C:=C_{1/2}\norm{\ift p}_{L^2(A)}^2\).
To get a small norm instead, we use orthogonal averaging to reduce the norm. Define
\[
    f=\frac1N\sum_{j=J}^{J+N-1}f_j,
\]
with \(J\) so that \(\sup_K|f_j|<\varepsilon\) for every \(j\geq J\) and
\(N\) so that \(C/N<\varepsilon^2\). Then still \(f(0)=1\),
\(f|_\Lambda=0\), and \(\sup_K|f|<\varepsilon\). Finally, the disjointness
of the Fourier supports yields
\[
    \norm{f}_{1/2,S}^2
    =\frac1{N^2}\sum_{j=J}^{J+N-1}\norm{f_j}_{1/2,S}^2
    \leq\frac{C}{N}<\varepsilon^2.
\]
This proves the endpoint case. Translating back completes the proof.
\end{proof}

\subsection{The general case}
To obtain the missing part of Theorem \ref{thm:sobolev}, we prove:
\begin{theorem}\label{thm:sobolev-negative}
Let \(0\leq\alpha\leq1/2\). No uniformly discrete set
\(\Lambda\subseteq\R\) is a uniqueness set for
\(PW^{(\alpha)}_S\cap C(\R)\). More precisely, for every
\(x_0\notin\Lambda\) there exists \(f\in PW^{(\alpha)}_S\cap C(\R)\) such
that \(f(x_0)=1\) and \(f|_\Lambda=0\).
\end{theorem}

\begin{proof}
Fix a uniformly discrete set \(\Lambda\subseteq\R\) and
\(x_0\notin\Lambda\). If \(\Lambda\) is finite, Proposition~\ref{prop:sobolev-peak}, applied with \(y_0=x_0\), \(K=\varnothing\), and
\(\varepsilon=1\), gives the required function. Assume then that \(\Lambda\)
is infinite, and enumerate it as \(\Lambda=\{\lambda_1,\lambda_2,\ldots\}\),
with \(|\lambda_n|\to\infty\). Set
\(K_n=[-n,n]\setminus(\lambda_n-1,\lambda_n+1)\), so that
\(\lambda_n\notin K_n\) and every compact subset of \(\R\) lies in
\(K_n\) for all sufficiently large \(n\).

By Proposition~\ref{prop:sobolev-peak}, choose
\(f_0\in PW^{(\alpha)}_S\cap C(\R)\) with \(f_0(x_0)=1\). Arguing by
induction, suppose that \(f_{n-1}(x_0)=1\) and
\(f_{n-1}(\lambda_j)=0\) for \(j<n\), and put
\(a_n=-f_{n-1}(\lambda_n)\) and
\(E_n=\{x_0,\lambda_1,\ldots,\lambda_{n-1}\}\).
Proposition~\ref{prop:sobolev-peak} gives
\(p_n\in PW^{(\alpha)}_S\cap C(\R)\) such that
\[
    p_n(\lambda_n)=1,
    \quad p_n|_{E_n}=0,
    \quad
    \max\left\{\norm{p_n}_{\alpha,S},
        \sup_{x\in K_n}|p_n(x)|\right\}
    <\frac{2^{-n}}{1+|a_n|}.
\]
Therefore \(f_n=f_{n-1}+a_np_n\) satisfies
\[
    \begin{gathered}
        f_n(x_0)=1,\quad f_n(\lambda_j)=0\quad(j\leq n),\\
        \norm{f_n-f_{n-1}}_{\alpha,S}<2^{-n},\quad
        \sup_{K_n}|f_n-f_{n-1}|<2^{-n}.
    \end{gathered}
\]

Since \(\sum_n2^{-n}<\infty\), the first estimate gives convergence in
\(PW^{(\alpha)}_S\), and the second gives local uniform convergence. We obtain
a limit \(f\in PW^{(\alpha)}_S\cap C(\R)\) with
\(f(x_0)=1\) and \(f|_\Lambda=0\), concluding the proof.
\end{proof}

\section{Higher-dimensional statements}\label{sec:higher-dim}

We collect the higher-dimensional analogues of our results from Section \ref{sec:intro}, and in the same order as the one-dimensional versions. To do so, we start by collecting some definitions.

For $\Lambda\subseteq\Rd$, write
$
    E(\Lambda)=\{e^{2\pi i\lambda\cdot x}:\lambda\in\Lambda\},
$
where $\lambda\cdot x$ denotes the Euclidean inner product. The set
$\Lambda$ is uniformly discrete if
\[
    \inf\{\|\lambda-\lambda'\|:\lambda,\lambda'\in\Lambda,
    \ \lambda\neq\lambda'\}>0,
\]
where $\| \cdot \|$ denotes the Euclidean distance in $\R^d$. The set $\Lambda$ has uniform density $D(\Lambda)$ if
\[
    \#\bigl(\Lambda\cap(y+[0,R)^d)\bigr)
    =D(\Lambda)R^d+o(R^d),\quad R\to\infty,
\]
uniformly in $y\in\Rd$. We use the Fourier conventions
\[
    \ft F(x)=\int_{\Rd}F(t)e^{-2\pi it\cdot x}\,dt,
    \quad
    \ift F(x)=\int_{\Rd}F(t)e^{2\pi it\cdot x}\,dt.
\]
For $\alpha\geq0$, let
$PW^{(\alpha)}_S(\Rd)$ denote the space of inverse Fourier transforms
$f=\ift F$, where $F\in L^2(\Rd)$ satisfies $F=0$ a.e.\ on $\Rd\setminus S$, and
\[
 \|f\|_{\alpha,S}^2:=\int_S\bigl(1+|t|^{2\alpha}\bigr)|F(t)|^2\,dt<\infty.
\]
Thus $PW_S=PW^{(0)}_S(\Rd)$, up to an equivalent norm.  In dimension $d>1$, periodic weak gaps mean that
$
 \bigl|(S+\Z^d)\cap[0,1]^d\bigr|<1.
$
We are ready to state the higher dimensional versions of our main results.

\vspace{0.3cm}

\begin{enumerate}
    \item Let $d\geq1$ and let $\alpha,\beta\in\Rd$ have the property that $1,\alpha_1,\dots,\alpha_d$ are linearly independent over over $\Q$,  $1+\alpha\cdot\beta,\beta_1,\dots,\beta_d$ are linearly independent over $\Q$, and $\|\beta\|_2<\tfrac12$.
For $n\in\Z^d$, define
\[
 \delta_n=\beta\Bigl(\{n\cdot\alpha\}-\frac12\Bigr),
 \quad
 \lambda_n=n+\delta_n,
 \quad
 \Lambda=\Lambda_{\alpha,\beta}=\{\lambda_n:n\in\Z^d\}.
\]
Then $\Lambda$ is uniformly discrete and has uniform density $D(\Lambda)=1$. For every measurable $S\subseteq\Rd$ with $|S|<1$ and every
$f\in L^1(S)$, the conditions
\[
 \int_S f(x)e^{-2\pi i\lambda\cdot x}\,dx=0,
 \quad\lambda\in\Lambda,
\]
imply that $f=0$ almost everywhere.  In particular, $E(\Lambda)$ is
complete in $L^p(S)$ for every $p\in[1,\infty)$ and every measurable
$S\subseteq\Rd$ with $|S|<1$.
\item Let $d\geq1$ and let $\Lambda=\{n+\delta_n:n\in\Z^d\}$, where
$\delta_n\in\Rd$ and $\delta_n\to0$ as $|n|\to\infty$.  For any
$\varepsilon>0$, there exists a measurable $S\subseteq\Rd$ such that
$|S|<\varepsilon$ and $E(\Lambda)$ is not complete in $L^2(S)$.
\item Let $d\geq1$, let $\alpha=(\alpha_1,\dots,\alpha_d)\in\Rd$ be such that
$1,\alpha_1,\dots,\alpha_d$ are linearly independent over $\Q$, let
$v\in[0,1]$, and define
\[
 \Lambda_v=\{n\in\Z^d:\{n\cdot\alpha\}\in[1-v,1)\}.
\]
Then $\Lambda_v$ is uniformly distributed with uniform density $D(\Lambda_v) = v$.
Moreover, for every measurable
$S\subseteq[0,1]^d$ with $|S|<v$ and every $f\in L^1(S)$, the
condition
\[
 \int_S f(x)e^{-2\pi ix\cdot\lambda}\,dx=0,
 \quad\lambda\in\Lambda_v,
\]
implies that $f=0$ almost everywhere.  In particular, the exponential
system $E(\Lambda_v)=\{e^{2\pi i\lambda\cdot x}:\lambda\in\Lambda_v\}$
is complete in $L^p(S)$ for every $p\in[1,\infty)$ and every
measurable $S\subseteq[0,1]^d$ with $|S|<v$.
\item Let $d\geq1$, let $A\subseteq[0,1]^d$ be measurable with $0<|A|<1$,
and set $S=A+\Z^d$.  If $0\leq\alpha\leq d/2$, then every uniformly
discrete set $\Lambda\subseteq\Rd$ is a nonuniqueness set for
$PW^{(\alpha)}_S(\Rd)\cap C(\Rd)$. On the other hand, if $\alpha > d/2$, then $PW^{(\alpha)}_S(\Rd)$ admits a uniformly discrete uniqueness set.
\item 
In particular, for every $d\geq1$, we find: For every measurable spectrum of the form
$S=A+\Z^d$, with $A$ is an (4), $PW_S\cap C(\Rd)$ has no
uniformly discrete uniqueness set. One may take $A=[0,\tfrac12]+\Z^d$.
\end{enumerate}

\section{Formalization in Lean}
\label{sec:lean}

This section discusses the Lean verification of items~(1)-(3) and the nonuniqueness conclusions in items~(4)-(5) of Section~\ref{sec:higher-dim}. The Lean formalization is available at the accompanying repository 
\begin{center}   \url{https://github.com/EnricFloritMath/UniversalCompleteness}.
\end{center}
In the repository, four \lean{Showcase.lean} files present the checked statements. Our goal was to make them accessible to mathematicians who are not well-versed in Lean.

The formalization relies on the library Mathlib. We now explain the results of Section~\ref{sec:higher-dim} in Lean.

\subsection{Universal completeness}
\label{subsec:lean-universal}

We start by formalizing the first theorem of Section \ref{sec:higher-dim} in Lean.
To do so, we recall that the frequencies in (1) of Section \ref{sec:higher-dim} are a bounded perturbation of $\Z^d$.
The parameters $\alpha,\beta\in\R^d$ satisfy two distinct rational
independence conditions: that of $1,\alpha_1,\ldots,\alpha_d$ and that of
$1+\alpha\cdot\beta,\beta_1,\ldots,\beta_d$.
We first write the dot product and the exponential
$x \mapsto e(\xi,x) = e^{2\pi i\xi\cdot x}$ with respect to the frequency $\xi \in \R^n$ in Lean.
\begin{leancode}
    def dot (x y : ℝᵈ) : ℝ := ∑ i, x i * y i
    def e (ξ x : ℝᵈ) : ℂ :=
      Complex.exp (((2 * π : ℝ) : ℂ) * I * (dot ξ x : ℂ))
\end{leancode}
Here \lean{def} introduces a definition. The declaration \lean{def dot}
takes two vector inputs, \lean{x} and \lean{y}, and returns a real number. The input and output types appear after the colons. The sum runs over
all $d$ coordinates. 
Next, we define the system of exponentials $E(\Lambda)$. In Lean, this system is defined as follows.
\begin{leancode}
    def E (Λ : Set ℝᵈ) : Set (ℝᵈ → ℂ) := e '' Λ
\end{leancode}
The term $\mathrm{e} \;''\; \Lambda$ denotes the image of $\Lambda$ under $\xi \mapsto e(\xi,\cdot)$, getting a set of exponential functions.
The two arithmetic assumptions given in the target statement are encoded in the the definition \lean{Independent}.
\begin{leancode}
    def Independent (α : ℝᵈ) : Prop :=
      ∀ (q₀ : ℚ) (q : Fin d → ℚ),
        (q₀ : ℝ) + ∑ i, (q i : ℝ) * α i = 0 → q₀ = 0 ∧ ∀ i, q i = 0.
\end{leancode}
The definition \lean{Nonresonant} is analogous and corresponds to the statement that $1+\alpha\cdot\beta, \beta_0,...\beta_{d-1}$ are linearly independent over $\mathbb Q$. 
We then define the Euclidean norm via
\begin{leancode}
    def norm₂ (x : ℝᵈ) : ℝ := √(∑ i, x i ^ 2). 
\end{leancode}
We are now ready to define the perturbation $(\delta_n)_{n \in \Z^d}$ as well as the set $\Lambda_{\alpha,\beta}$.
\begin{leancode}
    def δ (α β : ℝᵈ) (n : ℤᵈ) : ℝᵈ :=
      fun i => (Int.fract (∑ j, (n j : ℝ) * α j) - 1 / 2) * β i
    def Λ (α β : ℝᵈ) : Set ℝᵈ :=
      Set.range (fun n : ℤᵈ => fun i => (n i : ℝ) + δ α β n i)
\end{leancode}
The notation \lean{Set.range} takes all $n\in\Z^d$. Thus, the definition \lean{def} $\Lambda \; (\alpha \; \beta\;:\;\R^d)$ corresponds exactly to the set $\Lambda_{\alpha,\beta}=\{n+\delta_n:n\in\Z^d\}$, with $\delta_n=\beta(\{n\cdot\alpha\}-1/2)$.
Next, we define the notions of uniform discreteness and uniform density, using intersections of $\Lambda$ with translates of the half open cube $[0,R)^d$.
\begin{leancode}
    def UniformlyDiscrete (Λ : Set ℝᵈ) : Prop :=
      ∃ ρ > (0 : ℝ), ∀ x ∈ Λ, ∀ y ∈ Λ, x ≠ y → ρ ≤ norm₂ (x - y)
    def UniformDensity (Λ : Set ℝᵈ) (D : ℝ) : Prop :=
      (∀ (x : ℝᵈ) (R : ℝ), 0 < R →
        (Λ ∩ {y | ∀ i, x i ≤ y i ∧ y i < x i + R}).Finite) ∧
       ∀ ε > (0 : ℝ), ∃ R₀ > (0 : ℝ), ∀ (R : ℝ), R₀ ≤ R → ∀ (x : ℝᵈ),
         |((Λ ∩ {y | ∀ i, x i ≤ y i ∧ y i < x i + R}).ncard : ℝ) / R ^ d - D| ≤ ε
\end{leancode}
For \lean{UniformDensity}, we start by ensuring that the intersection is finite using \lean{.Finite} so that we can then consider its cardinality \lean{.ncard}. 
We use Mathlib's \lean{volume} for Lebesgue measure and
\lean{.restrict S} to restrict a measure to a measurable set $S$.
We then define what it means for $\Lambda$ to be a universal uniqueness set for $L^1(S)$ for every $S$ with measure bounded by $v$.
\begin{leancode}
    def L1Uniqueness (Λ : Set ℝᵈ) (v : ℝ) : Prop :=
      ∀ S : Set ℝᵈ, MeasurableSet S → volume S < ENNReal.ofReal v →
        ∀ f : ℝᵈ → ℂ, Integrable f (volume.restrict S) →
          (∀ g ∈ E Λ, (∫ x in S, f x * star (g x)) = 0) →
          f =ᵐ[volume.restrict S] 0
\end{leancode}
The above definition states that for every $S$ with $|S|<v$, if $f\in L^1(S)$ is such that $\int_S f(x) e^{-2\pi i \lambda x}\, dx = 0$ for every $\lambda\in \Lambda$ (here, \lean{star} is the complex conjugate), then $f=0$ almost everywhere on $S$. 

We similarly define completeness of the exponential system using the following Lean definition.
\begin{leancode}
    def Complete (Λ : Set ℝᵈ) (S : Set ℝᵈ) (p : ℝ≥0∞) (hp : 1 ≤ p) : Prop :=
      letI : Fact (1 ≤ p) := ⟨hp⟩
      Dense (↑(Submodule.span ℂ
        {g : Lp ℂ p (volume.restrict S) | ∃ f ∈ E Λ, g =ᵐ[volume.restrict S] f}) :
        Set (Lp ℂ p (volume.restrict S)))
\end{leancode}
The braces select $L^p(S)$ classes agreeing almost everywhere with an exponential in $E(\Lambda)$. The term \lean{Submodule.span} corresponds to finite complex linear combinations, and \lean{Dense} takes their closure in
the $L^p$ norm.

We note that (\lean{hp :} $1\le p$) is not a hypothesis: whenever we use \lean{Complete}, we need to provide a \textit{proof} of the fact that $1\leq p$. The term \lean{letI : Fact} $(1 \le p)$ \lean{:=} $\langle$\lean{hp}$\rangle$ turns this proof into a hypothesis so that the rest of the machinery of Mathlib's \lean{Lp} can be used. 
Consequently, we define universal completeness, with respect to a density parameter $v\in \mathbb R$, as follows:
\begin{leancode}
    def UniversallyComplete (Λ : Set ℝᵈ) (v : ℝ) : Prop :=
      ∀ S : Set ℝᵈ, MeasurableSet S → volume S < ENNReal.ofReal v →
        ∀ (p : ℝ≥0∞) (hp : 1 ≤ p), p < ∞ → Complete Λ S p hp
\end{leancode}
This definition represents exactly that $E(\Lambda)$ is complete in $L^p(S)$ for every $|S|<v$.  
Lastly, we can state the theorem in item (1) of Section \ref{sec:higher-dim}.
\begin{leancode}
    theorem universal_completeness
      (hd : 0 < d) (α β : ℝᵈ)     
      (hα : Independent α) (hβ : Nonresonant α β) (hsmall : norm₂ β < 1 / 2) :
      UniformlyDiscrete (Λ α β) ∧
      UniformDensity (Λ α β) 1 ∧
      L1Uniqueness (Λ α β) 1 ∧
      UniversallyComplete (Λ α β) 1 := by
\end{leancode}
The Lean statement of this theorem starts by fixing $d\geq 1$, and $\alpha, \beta\in\mathbb R^d$. We then assume linear independence over $\mathbb Q$ of $\lbrace 1, \alpha_0, ..., \alpha_{d-1}\rbrace$ and  $\lbrace 1+\alpha\cdot\beta, \beta_0, ..., \beta_{d-1}\rbrace$. The conclusions of the theorem are listed in the following order: $\Lambda\; \alpha \;\beta$ is uniformly discrete, has uniform density $1$, is a uniqueness set for every $S$ with measure $<1$, and $E(\Lambda \;\alpha \;\beta)$ is complete in $L^p(S)$, $1\leq p <\infty$, for every $|S|<1$. 

\subsection{Null perturbations}
\label{subsec:lean-null}
The second theorem is item (2) in Section~\ref{sec:higher-dim} and concerns lattice perturbations by a sequence $\delta_n$ that tend to zero as $|n|\to\infty$. In order to state this in Lean, we use \lean{integerEmbed} which embeds $\mathbb Z^d$ into $\mathbb R^d$. We write the perturbation as a function $\delta : \Z^d \to \R^d$ and call the corresponding definition \lean{NullPerturbation}. The associated set $\Lambda$ constructed from this perturbation is defined by \lean{def} $\Lambda \;  (\delta \;:\; \Z^d \to \R^d)$.
\begin{leancode}
    def NullPerturbation (δ : ℤᵈ → ℝᵈ) : Prop :=
      ∀ η > (0 : ℝ), ∃ N : ℕ, ∀ n : ℤᵈ,
        (N : ℝ) ≤ norm₂ (integerEmbed n) → norm₂ (δ n) < η
    def Λ (δ : ℤᵈ → ℝᵈ) : Set ℝᵈ :=
      Set.range (fun n : ℤᵈ => fun i => (n i : ℝ) + δ n i)
\end{leancode}
The latter definitions allow us to state the main theorem, which is the negation of \lean{Complete} (defined as in the previous subsection) in the particular case $p=2$:
\begin{leancode}
    theorem null_perturbations
      (hd : 0 < d) (δ : ℤᵈ → ℝᵈ) (hδ : NullPerturbation δ)
      (ε : ℝ) (hε : 0 < ε) :
      ∃ S : Set ℝᵈ, MeasurableSet S ∧ volume S < ENNReal.ofReal ε ∧
        ¬ Complete (Λ δ) S 2 (by norm_num) := by
\end{leancode}
Here we fix a dimension $d$, a sequence $\delta$ that decays to $0$, and a constant $\varepsilon>0$. Then, the statement claims that we can obtain a measurable set $S$ with  measure $|S|<\varepsilon$ for which $\{ n+\delta_n : n \in \Z^d \}$ is not complete in $L^2(S)$.

\subsection{Integer frequencies}
\label{subsec:lean-integer}
For the result on integer frequencies, we have that $\Lambda_v$ is a subset of $\Z^d$, selected according to whether the fractional part of $n\cdot\alpha$ belongs to $[1,v,1)$. The half-open interval $[1-v,1)$ is written in Lean by \lean{Set.Ico (1 - v) 1} and the set $\Lambda_v$ is then defined as follows.
\begin{leancode}
    def Λ (α : ℝᵈ) (v : ℝ) : Set ℤᵈ :=
      {n | Int.fract (∑ i, (n i : ℝ) * α i) ∈ Set.Ico (1 - v) 1}
\end{leancode} 
The theorem in question concerns measurable subsets $S\subset [0,1]^d$ of the
closed unit cube, with measure $|S|<v$. The cube is defined separately for readability.
\begin{leancode}
def unitCube : Set ℝᵈ := {x | ∀ i, x i ∈ Set.Icc (0 : ℝ) 1}
\end{leancode}
We also define the appropriate notions of universal uniqueness and completeness, \lean{L1UniquenessOnUnitCube} and \lean{UniversalCompletenessOnUnitCube}: these are just as \lean{L1Uniqueness} and \lean{UniversalCompleteness}, except that we have the additional restriction $S \subset [0,1]^d$ for the sets considered.
The theorem claims that $\Lambda_v$ has density $v$, and that the corresponding universal completeness claim holds. In Lean the theorem is stated as follows.
\begin{leancode}
    theorem integer_frequencies
      (hd : 0 < d) (α : ℝᵈ) (hα : Independent α)
      (v : ℝ) (hv : v ∈ Set.Icc (0 : ℝ) 1) :
      UniformDensity (Λ α v) v ∧ L1UniquenessOnUnitCube (Λ α v) v ∧
      UniversallyCompleteOnUnitCube (Λ α v) v := by
\end{leancode}
In the above theorem, we fix $d> 0$, $\alpha\in \mathbb R^d$ such that $\alpha$ is \lean{Independent} and $v\in [0,1]$. The theorem then states the density claim of the theorem together with the uniqueness claim for $L^1(S)$ for every $S\subset [0,1]^d$ with $|S|<v$.  

\subsection{Sobolev nonuniqueness}
\label{subsec:lean-sobolev}

To formalize the corresponding Sobolev result, we first define $\R^d$ via the following Lean notation.
\begin{leancode}
    local notation "ℝᵈ" => EuclideanSpace ℝ (Fin d)
\end{leancode}
The periodization of a set $A$, namely $A+\Z$, is given by the following definition.
\begin{leancode}
    def periodicSpectrum (A : Set ℝᵈ) : Set ℝᵈ :=
      {ξ | ∃ k : ℤᵈ, ξ - WithLp.toLp 2 (fun i => (k i : ℝ)) ∈ A}
\end{leancode}
A comment on \lean{WithLp.toLp 2}: this allows us to reinterpret $k\; i$ as an element of the correct type (a function versus an element of the Euclidean space), so that the sum $\xi + k$ makes sense as an element of the Euclidean space $\R^d$.

The weight
$w:=1+|t|^{2\alpha}$ is defined next as $2$ for $\alpha =0$ and as $1+|t|^{2\alpha}$ otherwise. The Lean definition reads as follow.
\begin{leancode}
    def w (α : ℝ) (t : ℝᵈ) : ℝ :=
      1 + if α = 0 then 1 else Real.rpow ‖t‖ (2 * α)
\end{leancode}
We next define the set \lean{PWalpha} which consists of functions $f$ that have a Fourier
representative $F$ supported almost everywhere in $S$ and such that $F^2w$ is integrable. Moreover, we define the classical Paley-Wiener space $PW_S$ by requiring that $\alpha=0$.
\begin{leancode}
    def PWalpha (α : ℝ) (S : Set ℝᵈ) : Set (ℝᵈ → ℂ) :=
      {f | ∃ F : ℝᵈ → ℂ,
        (∀ᵐ t ∂volume, t ∉ S → F t = 0) ∧
        MemLp F 2 ((volume.restrict S).withDensity (fun t => ENNReal.ofReal (w α t))) ∧
        ∃ (hF : MemLp F 2 volume) (hf : MemLp f 2 volume),
          (��⁻ (hF.toLp F) : Lp ℂ 2 volume) = hf.toLp f}
    def PW (S : Set ℝᵈ) : Set (ℝᵈ → ℂ) := PWalpha 0 S
\end{leancode}
We notice that $\mathcal F^{-1}$ is Mathlib's notation for the inverse Fourier transform from $L^2$ to itself, the final equality in the definition of \lean{PWalpha} must be an equality between two elements of $L^2$. However,  $F,f$ are a priori not elements of this space, and therefore we ask for the existence of a proof, respectively \lean{hF} and \lean{hf}, that $F,f$ are in $L^2$. Then, $\mathcal F^{-1}$ is applied to $F$ seen as a member of $L^2$, where this is ensured by \lean{hF}. The same holds for $f$.  

We now state the theorem. For the range
$0\leq\alpha\leq d/2$, given $\Lambda$ uniformly discrete, it asserts the existence of one continuous $f$
whose value is $1$ at any prescribed $x_0\notin\Lambda$ and which vanishes at every point of $\Lambda$.
\begin{leancode}
    theorem sobolev_nonuniqueness
      (hd : 0 < d) (A : Set ℝᵈ)
      (hA : MeasurableSet A) (hAcube : A ⊆ unitCube)
      (hApos : 0 < volume A) (hAsmall : volume A < 1)
      (α : ℝ) (hα : 0 ≤ α ∧ α ≤ (d : ℝ) / 2)
      (Λ : Set ℝᵈ) (hΛ : UniformlyDiscrete Λ)
      (x₀ : ℝᵈ) (hx₀ : x₀ ∉ Λ) :
        ∃ f ∈ PWalpha α (periodicSpectrum A),
        Continuous f ∧ f x₀ = 1 ∧ ∀ ξ ∈ Λ, f ξ = 0 := by
\end{leancode}
We recall that we do not prove existence of a uniformly discrete uniqueness
set for $\alpha>d/2$, as this was already done in \cite{OlevskiiUlanovskii2017}.

\Needspace{23\baselineskip}
\subsection{Paley--Wiener nonuniqueness}
\label{subsec:lean-endpoint}

Setting $\alpha=0$ in the preceding theorem gives a nonzero continuous function in
$PW_{A+\Z^d}$ vanishing on a given uniformly discrete set. The formalization of the endpoint non-uniqueness theorem posed by Olevskii and Ulanovskii reads as follows.
\begin{leancode}
    theorem endpoint_nonuniqueness
      (hd : 0 < d) (A : Set ℝᵈ)
      (hA : MeasurableSet A) (hAcube : A ⊆ unitCube)
      (hApos : 0 < volume A) (hAsmall : volume A < 1) :
      ∀ Λ : Set ℝᵈ, UniformlyDiscrete Λ →
        ∃ f ∈ PW (periodicSpectrum A),
        Continuous f ∧ f ≠ 0 ∧ ∀ ξ ∈ Λ, f ξ = 0 := by
\end{leancode}
The theorem states precisely that given $A\subset[0,1]^d$ measurable, with $0<|A|<1$, and a uniformly discrete set $\Lambda$, there exists a nonzero, continuous $f \in PW_{A+\Z^d}$ which vanishes on $\Lambda$. This is exactly the nonuniqueness conclusion of item (5) in Section~\ref{sec:higher-dim}.

All Lean statements compile under Lean 4.31.0. Their five public theorems depend only on Lean's standard foundational axioms: \lean{propext},
\lean{Classical.choice}, and \lean{Quot.sound}. The
showcase validation record
documents the current statement and axiom checks. Independent proof checks
and their different scopes are recorded separately for (1), (2), (3) and (4). 

\section*{Acknowledgments}
S.~B. and E.~F. were supported by the European Research Council
under the Grant Agreement No~948029.

L.L.~is grateful to the Azrieli Foundation for the award of an Azrieli Fellowship and acknowledges the support of this research by ISF Grant No.~854/25.

\bibliographystyle{plain}
\bibliography{bibfile}

\end{document}